\documentclass[a4paper,12pt]{amsart}
\usepackage{times}
\usepackage{bm}
\usepackage{amssymb}
\usepackage{latexsym}
\usepackage{amsmath}
\usepackage{amscd}
\usepackage[dvipdfmx]{graphicx}
\usepackage[dvipdfmx]{pict2e} 
\usepackage{color}

\makeatletter
\def\Left#1#2\Right{\begingroup%
   \def\ts@r{\nulldelimiterspace=0pt \mathsurround=0pt}%
   \let\@hat=#1%
   \def\sht@im{#2}%
   \def\@t{{\mathchoice{\def\@fen{\displaystyle}\k@fel}%
          {\def\@fen{\textstyle}\k@fel}%
          {\def\@fen{\scriptstyle}\k@fel}%
          {\def\@fen{\scriptscriptstyle}\k@fel}}}%
   \def\g@rin{\ts@r\left\@hat\vphantom{\sht@im}\right.}%
   \def\k@fel{\setbox0=\hbox{$\@fen\g@rin$}\hbox{%
      $\@fen \kern.3875\wd0 \copy0 \kern-.3875\wd0%
      \llap{\copy0}\kern.3875\wd0$}}%
      \def\pt@h{\mathopen\@t}\pt@h\sht@im%
      \Right}%
\def\Right#1{\let\@hat=#1%
   \def\st@m{\mathclose\@t}%
   \st@m\endgroup}
\makeatother

\begin{document}

\title[On the structure of groups defined by Kim and Manturov]
{On the structure of groups defined by Kim and Manturov}

\author[C.-F.\ Nyberg-Brodda]{Carl-Fredrik Nyberg-Brodda}
\address{June E. Huh Center for Mathematical Challenges, Korea Institute for Advanced Study, 
85 Hoegi-ro, Dongdaemun-gu, Seoul 02455, 
Republic of Korea}
\email{cfnb@kias.re.kr}
\author[T.\ Sakasai]{Takuya Sakasai}
\address{Graduate School of Mathematical Sciences, 
The University of Tokyo, 
3-8-1 Komaba, 
Meguro-ku, Tokyo, 153-8914, Japan}
\email{sakasai@ms.u-tokyo.ac.jp}
\author[Y.\ Tadokoro]{Yuuki Tadokoro}
\address{Faculty of Science Division II, Department of Mathematics, 
Tokyo University of Science, 1-3 Kagurazaka, Shinjuku-ku, Tokyo 162-8601, Japan}
\email{tado@rs.tus.ac.jp}
\author[K.\ Tanaka]{Kokoro Tanaka}
\address{Department of Mathematics, 
Tokyo Gakugei University, 
4-1-1 Nukuikita-machi, 
Koganei-shi, Tokyo 184-8501, Japan}
\email{kotanaka@u-gakugei.ac.jp}

\date{\today}

\thanks{
The authors were partially supported by KAKENHI 
(No.21H00986, No.24K06751, No.21K03220, No.24K06740), 
Japan Society for the Promotion of Science, Japan.
}

\subjclass[2000]{Primary~20D15, Secondary~20F14, 57M27}
\keywords{pentagon relation, Coxeter group, Artin group, triangulation, coset enumeration}


\newtheorem{thm}{Theorem}[section]
\newtheorem{prop}[thm]{Proposition}
\newtheorem{lem}[thm]{Lemma}
\newtheorem{cor}[thm]{Corollary}
\theoremstyle{definition}
\newtheorem{definition}[thm]{Definition}
\newtheorem{example}[thm]{Example}
\newtheorem{remark}[thm]{Remark}
\newtheorem{problem}[thm]{Problem}
\newtheorem{conj}[thm]{Conjecture}
\renewcommand{\theproblem}{}

\newcommand{\gen}[1]{\langle #1 \rangle}

\newcommand{\Ker}{\mathop{\mathrm{Ker}}\nolimits}
\newcommand{\Hom}{\mathop{\mathrm{Hom}}\nolimits}
\renewcommand{\Im}{\mathop{\mathrm{Im}}\nolimits}

\newcommand{\Der}{\mathop{\mathrm{Der}}\nolimits}
\newcommand{\Out}{\mathop{\mathrm{Out}}\nolimits}
\newcommand{\Aut}{\mathop{\mathrm{Aut}}\nolimits}
\newcommand{\End}{\mathop{\mathrm{End}}\nolimits}
\newcommand{\Q}{\mathbb{Q}}
\newcommand{\Z}{\mathbb{Z}}
\newcommand{\R}{\mathbb{R}}
\newcommand{\C}{\mathbb{C}}

\newlength{\Width}
\newlength{\Height}
\newlength{\Depth}

\begin{abstract}
We study the structure of a series of groups 
$\Gamma_n^4$ defined by Kim and Manturov. We show that the group $\Gamma_n^4$ is 
finite for any $n \ge 6$ and in fact it is a $2$-step nilpotent $2$-group of order $2^{\binom{n}{3}}$. 
\end{abstract}

\renewcommand\baselinestretch{1.1}
\setlength{\baselineskip}{16pt}

\newcounter{fig}
\setcounter{fig}{0}

\maketitle

\section{Introduction}\label{sec:intro}
This is a continuation of the paper \cite{DTS} which studies 
a series of groups $\Gamma_n^4$ defined by 
Kim and Manturov in \cite{KM}. Set $[n]=\{1,2,\ldots, n\}$ for a positive integer $n$.  
For $n \ge 4$, 
the group $\Gamma_n^4$ is defined by the presentation below, where 
$(ijkl)$ is a symbol for an ordered quadruple of distinct integers $i,j,k,l \in [n]$. 
We use the notation $(ijkl)$ instead of $d_{ijkl}$ in the original paper \cite{KM} for visibility. 
\begin{definition}\label{def:gamma}
For $n \ge 4$, the group $\Gamma_n^4$ is defined by the following presentation: 

(Generators) $\mathcal{G}_n:=\{(ijkl) \mid \{i,j,k,l\}\subset [n],\ (i,j,k,l\text{: distinct})\}$

(Relations) There are four types of relations: 
\[\begin{array}{ll}
 (1)&  (ijkl)^2 =1; \\
 (2)&  (ijkl)(stuv)=(stuv)(ijkl),\ (|\{i,j,k,l\} \cap \{s,t,u,v\}| \le 2);\\
 (3)&  (ijkl)(ijlm)(jklm)(ijkm)(iklm)=1, \ (i,j,k,l,m \text{ distinct});\\
 (4)&  (ijkl)=(jkli)=(lkji).
\end{array}\]
\end{definition}
We call the relations (1) the {\it involutive relations}, (2) the {\it commutative relations}, 
(3) the {\it pentagon relations} and (4) the {\it dihedral relations}. 
Specifically, we call (3) for fixed $i,j,k,l,m$ the {\it pentagon relation for $\{i,j,k,l,m\}$}, 
where we respect the order of $i,j,k,l,m$. 

We refer to the paper \cite{KM} and the book \cite[Chapter 15]{M_book} for the geometric 
background of the group $\Gamma_n^4$, where the respective authors consider special kinds of 
configurations of $n$ points in a plane pertaining to Delaunay triangulations of a surface. The group $\Gamma_n^4$ is defined as the {\it target} of invariants for motions of configurations of points, where each of the above generators represents 
a Whitehead move (flip, Pachner move in two dimensional case) 
among Delaunay triangulations. See Figure \ref{fig:geom} for a geometric interpretation of the relations.  
Note, however, that the group $\Gamma_n^4$ itself does not need any 
geometric objects, and will be regarded as an abstract and combinatorial object. 

\begin{figure}[htbp]
\begin{center}
\includegraphics[width=0.8\textwidth]{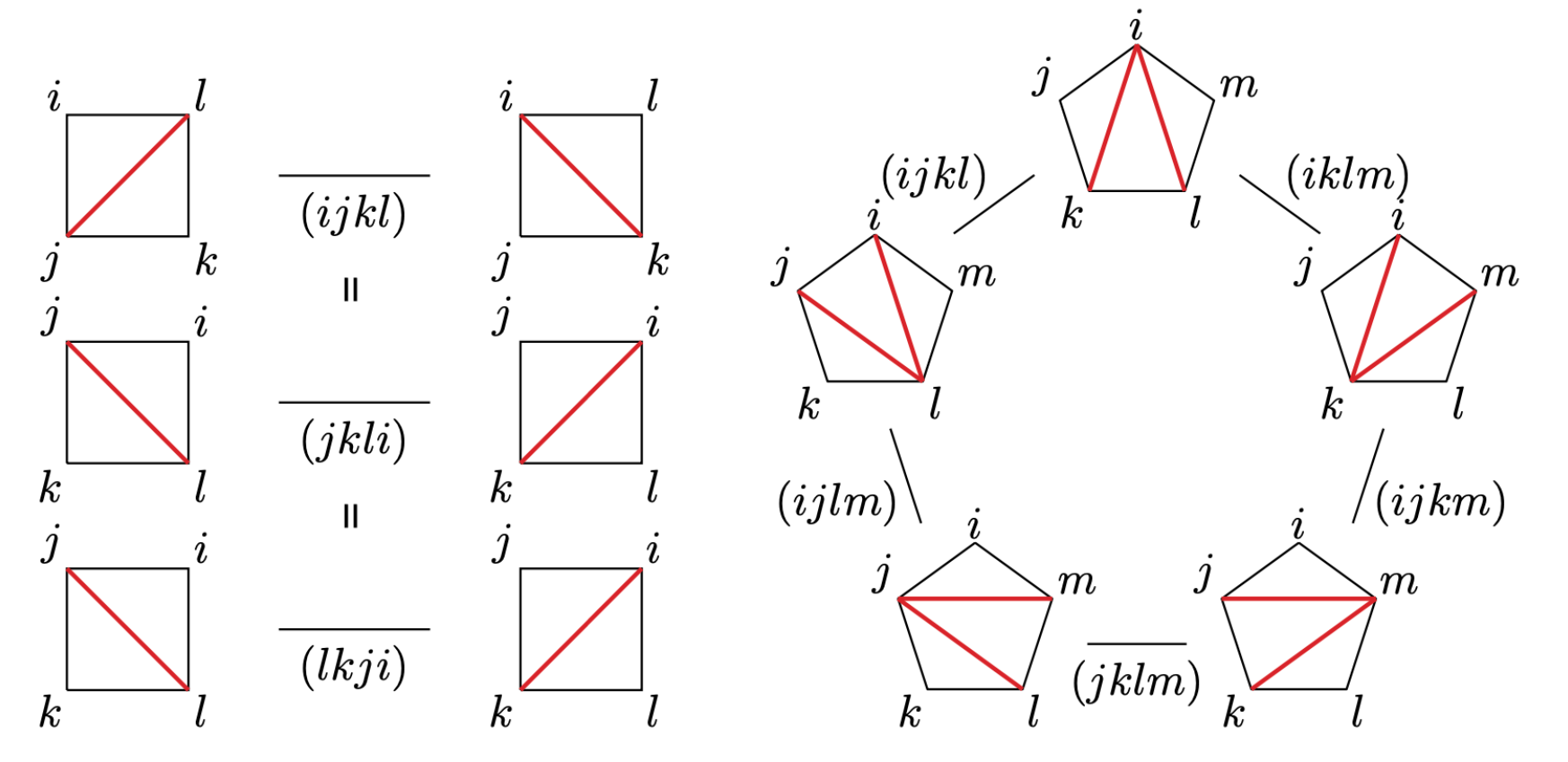}
\caption{Graphical meaning of the relations of $\Gamma_n^4$}
\label{fig:geom}
\end{center}
\end{figure}

In the aforementioned references, potential connections between the group $\Gamma_n^4$ and other geometric objects, 
such as invariants of braids and links, are discussed. To make these connections actual, it seems necessary to first elucidate the structure of $\Gamma_n^4$, but the structure of the group $\Gamma_n^4$ itself has not been sufficiently studied. In our previous paper \cite{DTS}, starting from the presentation, we gave a minimal generating set and determined 
the abelianization $H_1 (\Gamma_n^4)$ of $\Gamma_n^4$ by a purely group-theoretical argument. 
Continuing this line of research, we will show the following in this paper: 
\begin{thm}\label{thm:main}
For any $n \ge 6$, the group $\Gamma_n^4$ is a finite $2$-step nilpotent $2$-group of order $2^{\binom{n}{3}}$. 
In fact, we have a central extension 
\[0 \longrightarrow \mathbb{Z}/2\mathbb{Z} \longrightarrow 
\Gamma_n^4\longrightarrow H_1 (\Gamma_n^4) 
\longrightarrow 1.\]
\end{thm}

The group $\Gamma_4^4$ is isomorphic to the free product of 
the three copies of $\mathbb{Z}/2\mathbb{Z}$ generated by 
$\{(1234), (1324), (1243)\}$. 
Also, we proved in \cite{DTS} that $\Gamma_5^4$ is an infinite, non-abelian group not satisfying Property (T). 
Theorem \ref{thm:main} addresses the general case and shows a strong contrast with the previous cases. 
In particular, the natural map $\Gamma_5^4 \to \Gamma_6^4$ is not injective. 
We will give an explicit description for the above central extension through a matrix representation of $\Gamma_n^4$. 

{\it Added\/}: After posting the first version of this paper on arXiv, 
the authors were informed from Professor Manturov that 
in \cite{Styrt} Styrt had previously shown that $\Gamma_n^4$ is a nilpotent finite $2$-group with $4$-torsion 
and that the commutator subgroup $[\Gamma_n^4, \Gamma_n^4]$ is central 
for $n \ge 7$. An upper bound of the order was also given. The present paper  
derives a smaller upper bound for $n \ge 6$ by an independent argument, 
and in fact shows that it gives the precise order. 
 
This paper is organized as follows. 
In Section \ref{sec:DTS1}, we recall a number of facts on the group $\Gamma_n^4$ 
studied in our previous paper. 
The proof of Theorem \ref{thm:main} is given in Sections \ref{sec:finite} and \ref{sec:nonabelian}. 
First, we show that the group $\Gamma_n^4$ is finite in Section \ref{sec:finite}. In fact, 
we narrow down the potential forms of $\Gamma_n^4$ to two candidates.
Then, in Section \ref{sec:nonabelian}, 
we show that the group $\Gamma_n^4$ is not an abelian group by constructing 
an explicit matrix representation, which determines the structure of $\Gamma_n^4$. 
Section \ref{sec:computer} presents another proof of Theorem \ref{thm:main}, 
which is partially aided by a computer. 
Finally, in Section \ref{sec:remarks}, we discuss relevant topics and  
potential future research directions concerning the study of $\Gamma_n^4$ and its relatives.  

Remarks on notations are in order. 
The commutator $[a,b]$ means $a b a^{-1} b^{-1}$. 
The cardinality of a finite set $X$ is denoted by $|X|$. 
We denote by $\mathfrak{S}_n$ the symmetric group of degree $n$. 
The element $(ij) \in \mathfrak{S}_n$ with $i \neq j$ represents the transposition of $i$ and $j$. 
The unit element in a group is denoted by $1$ and the identity matrix of size $k$ is denoted by $I_k$.

\section{A minimal generating set of the group $\Gamma_n^4$}\label{sec:DTS1}
As mentioned in the introduction, some results about the groups $\Gamma_n^4$ were proved in \cite{DTS}, which we now recall.  
\begin{thm}[Minimal generating set {\cite[Theorem 3.1]{DTS}}]\label{thm:generating_gamma}
For $n \ge 4$, the group $\Gamma_n^4$ is generated by the set $\Lambda_n$ consisting of 
\begin{itemize}
\item[(G1)] $(1 2 3 l)$ with $4 \le l \le n$, 

\item[(G2)] $(1 j 2 l)$ with $3 \le j < l \le n$, 

\item[(G3)] $(1 j l k )$ with $2 \le j < k < l \le n$. 
\end{itemize}
The set $\Lambda_n$ has $N_n=\binom{n}{3}-1$ elements, and the group $\Gamma_n^4$ cannot be generated by fewer than $N_n$ elements. 
\end{thm}

We see from the defining presentation of $\Gamma_n^4$ that 
the symmetric group $\mathfrak{S}_n$ of degree $n$ acts on $\Gamma_n^4$ from the left by 
\[\sigma (ijkl):=( \sigma(i)\,\sigma(j)\,\sigma(k)\,\sigma(l) )\]
for $\sigma \in \mathfrak{S}_n$. 

By the dihedral relation, 
each generator $(ijkl) \in \mathcal{G}_n$ is equal to a unique element $(i'j'k'l')$ in the set 
\[\mathcal{G}_n^d:=\{ (a b c d) \mid a=\min \{a,b,c,d\},  b< d\}\]
as an element of $\Gamma_n^4$. 
We call $(i'j'k'l')$ the {\it minimal expression} of $(ijkl)$. 
By definition, we have 
\[\Lambda_n \subset \mathcal{G}_n^d \subset \mathcal{G}_n.\] 

Let $[n]_k$ denote the set of $k$ elements subsets of $[n]$ and 
$(\mathbb{Z}/2\mathbb{Z})[n]_k$ the $(\mathbb{Z}/2\mathbb{Z})$-vector space 
spanned by the set $[n]_k$. 
Consider the homomorphisms 
$\Phi_k^{(2)} \colon \Gamma_n^4 \to (\mathbb{Z}/2\mathbb{Z})[n]_k$ for $k=2,3$ defined by
\begin{align*}
\Phi_3^{(2)} ((ijkl)) &= \{i, j, k\} + \{i, j, l\}+\{i,k,l\} + \{j, k, l\}, \\
\Phi_2^{(2)} ((ijkl)) &= \{i, k\} + \{j, l\}. 
\end{align*}
\noindent
Then we have the following. 
\begin{thm}[Abelianization {\cite[Theorem 3.2]{DTS}}]\label{thm:abelianization}
For $n \ge 4$, the image of the homomorphism 
\[\Phi_3^{(2)} \oplus \Phi_2^{(2)} \colon 
\Gamma_n^4 \longrightarrow (\mathbb{Z}/2\mathbb{Z}) [n]_3 \oplus 
(\mathbb{Z}/2\mathbb{Z}) [n]_2 \cong 
(\mathbb{Z}/2\mathbb{Z})^{\binom{n}{3}+\binom{n}{2}}\]
is isomorphic to $(\mathbb{Z}/2\mathbb{Z})^{N_n}$, which gives the abelianization $H_1 (\Gamma_n^4)$ 
of $\Gamma_n^4$. The minimal generating set $\Lambda_n$ is a basis of $H_1 (\Gamma_n^4)$.  
\end{thm}

\section{The group $\Gamma_n^4$ is finite for $n \ge 6$}\label{sec:finite}

We show that the group $\Gamma_n^4$ is finite for $n \ge 6$. To do this, we first define 
\[\overline{\Gamma_n^4}= (\Gamma_n^4 \ast \langle c \rangle)/
\Left< c = [(ijkl),(stuv)] \mid |\{i,j,k,l\} \cap \{s,t,u,v\}| =3 \Right> ,\]
where $\Left< \cdot \Right>$ means the normal closure. That is, 
we collapse all the commutators $[(ijkl),(stuv)]$ with $|\{i,j,k,l\} \cap \{s,t,u,v\}| =3$ into a single element $c$. 
The natural map $\Gamma_n^4 \to \overline{\Gamma_n^4}$ is surjective.  
The element $c$ is a unit or of order $2$ in $\overline{\Gamma_n^4}$ 
because $c=[(1234),(1235)]=[(1235),(1234)]^{-1}=c^{-1} \in \overline{\Gamma_n^4}$. 
In fact, the latter holds and its proof will be given in the next section. 
The action of $\mathfrak{S}_n$ on $\Gamma_n^4$ induces that on $\overline{\Gamma_n^4}$, which 
fixes $c \in \overline{\Gamma_n^4}$. Now the commutative relation says that $[c, (3456)]=[[(1234),(1235)],(3456)]=1$. 
The action of $\mathfrak{S}_n$ on $\Gamma_n^4$ shows that $c$ is in the center of $\overline{\Gamma_n^4}$. Hence, assuming that 
$c \neq 1 \in \overline{\Gamma_n^4}$, we have a central extension 
\[0 \longrightarrow \mathbb{Z}/2\mathbb{Z} \longrightarrow 
\overline{\Gamma_n^4} \longrightarrow H_1 (\overline{\Gamma_n^4}) 
\longrightarrow 1, \]
where the generator of $\mathbb{Z}/2\mathbb{Z} = [\overline{\Gamma_n^4}, \overline{\Gamma_n^4}]$ 
is $c$. Since the surjection $\Gamma_n^4 \to \overline{\Gamma_n^4}$ induces an isomorphism on the abelianization, 
we see that 
\[H_1 (\overline{\Gamma_n^4}) \cong H_1 (\Gamma_n^4) \cong (\mathbb{Z}/2\mathbb{Z})^{N_n}\]
and that $\overline{\Gamma_n^4}$ is a finite $2$-step nilpotent $2$-group of order $2^{N_n+1}=2^{\binom{n}{3}}$. 
The finiteness of $\Gamma_n^4$ is derived from the following. 
\begin{prop}\label{prop:finitequotient}
Let $n \ge 6$. If $|\{i,j,k,l\} \cap \{s,t,u,v\}| =3$, we have 
\[[(ijkl), (stuv)]=[(1234), (1235)] \in \Gamma_n^4.\]
Therefore the natural surjection $\Gamma_n^4 \twoheadrightarrow \overline{\Gamma_n^4}$ is an isomorphism.
\end{prop}

To prove this proposition, we need a series of lemmas on equalities in $\Gamma^4_n$. 

\begin{lem}\label{lem:importantrel}
Let $i,j,k,l,m,p$ be distinct integers in $[n]$ with $n \ge 6$. 
The equalities
\begin{align*}
(1) & \quad [(iklp),(iklm)] = [(iklp),(ijkl)]\\
(2) & \quad [(ikpl),(iklm)] = [(ikpl),(ijkl)]
\end{align*}
\noindent
hold.
\end{lem}
\begin{proof}
By the pentagon relation for $\{i,j,k,l,m\}$, we have 
\[(ijlm)(jklm)(ijkm)=(ijkl)(iklm).\]
Since $(iklp)$ commutes with each of $(ijlm), (jklm), (ijkm)$, we see that 
$(ijkl)(iklm)$ commutes with $(iklp)$ . Then 
\begin{align*}
1 &= (ijkl)(ijkl)(iklm)(iklm)=(ijkl) (iklp) \underline{(iklp)(ijkl)(iklm)}(iklm)\\
&=(ijkl) (iklp) \underline{(ijkl)(iklm) (iklp)} (iklm)
=[(ijkl), (iklp)] [(iklp), (iklm)], 
\end{align*}
which proves the equality (1).  
Since $(ijkl)(iklm)$ also commutes with $(ikpl)$, we may apply the above argument 
words by words after replacing $(iklp)$ by $(ikpl)$ to show the equality (2). 
\end{proof}

\begin{lem}\label{lem:comm3}
If $|\{i,j,k,l\} \cap \{s,t,u,v\}| =3$, say $\{i,j,k,l\} \cap \{s,t,u,v\}=\{x,y,z\}$, then 
the commutator  
$[(ijkl), (stuv)]$ depends only on the set $\{x,y,z\}$. 
Hence the element 
\[c_{xyz}:=[(ijkl), (stuv)]\]
is well-defined and $c_{xyz}^2=1$ holds. 
\end{lem}

\begin{proof}
We consider the case where $\{i,j,k,l\} \cap \{s,t,u,v\}=\{1,2,3\}$. 
The general case is derived by considering the action of $\mathfrak{S}_n$. 
Note that up to the dihedral relation, the generators in $\mathcal{G}_n$ 
having $\{1,2,3\}$ are $(123k), (12k3), (132k)$ for $4 \le k \le n$. 

We use the equalities in Lemma \ref{lem:importantrel}. Applying (1) for $(i,j,k,l,m,p)=(1,6,2,3,5,4)$ and 
$(i,j,k,l,m,p)=(1,4,3,2,5,6)$ 
we have 
\begin{align*}
 [(1234), (1235)]&=[(1234), (1623)]=[(1234), (1326)], \\
 [(1326), (1325)]&=[(1326), (1432)]=[(1326), (1234)]. 
 \end{align*}
\noindent
From them, we have 
\begin{equation}
[(1234), (1235)]=[(1325), (1326)]. \label{a1} \tag{a$_{1}$}
\end{equation} 
Then we have 
\[[(1234), (1235)]=[(3214),(3215)] = [(3214), (3621)]
=[(1234), (1263)]. \label{a2} \tag{a$_{2}$}\]
by applying (1)  for $(i,j,k,l,m,p)=(3,6,2,1,5,4)$ to the second equality. 
Also, (2) for $(i,j,k,l,m,p)=(2,5,1,3,6,4)$ gives 
\[[(1234), (1263)]=[(2143),(2136)] = [(2143), (2513)]
=[(1234), (1325)]. \label{a3} \tag{a$_{3}$}\]
and (1) for $(i,j,k,l,m,p)=(1,4,3,2,6,5)$ gives 
\[[(1325), (1326)]=[(1325),(1432)] = [(1325), (1234)]. \label{a4} \tag{a$_{4}$}\]
From \eqref{a2}, \eqref{a3} and \eqref{a4}, we see that 
\[[(1234), (1235)]=[(1326), (1325)]. \label{a5} \tag{a$_{5}$} \]
Now the equalities \eqref{a1} and \eqref{a5} imply that 
\[[(1325), (1326)]=[(1234), (1235)]=[(1326), (1325)]=[(1325), (1326)]^{-1}.\]
Hence 
\[[(1234), (1235)]^2=[(1325), (1326)]^2=1.\]
Applying $(56) \in \mathfrak{S}_n$ to \eqref{a5}, we have 
\[[(1234), (1236)]=[(1325), (1326)]. \]
Combining with \eqref{a5}, we see that 
\[[(1234), (1235)]=[(1234), (1236)]. \label{a6} \tag{a$_{6}$}\]

For distinct $k, l \in \{7,8,\ldots, n\}$, applying $(6l), 
(6k)(45)(4l), (45), (6l)(45), (46), (5k)(46) \in \mathfrak{S}_n$ 
to \eqref{a6} gives
\begin{align*}
[(1234), (1235)]&=[(1234), (123l)],\\
[(123l), (1234)]&=[(123l), (123k)],\\
[(1235), (1234)]&=[(1235), (1236)],\\
[(1235), (1234)]&=[(1235), (123l)],\\
[(1236), (1235)]&=[(1236), (1234)],\\
[(1236), (123k)]&=[(1236), (1234)].
\end{align*}
\noindent
From these equalities, we see that 
\[[(1234), (1235)]=[(123k), (123l)]. \label{a7} \tag{a$_{7}$}\]
holds for all distinct $k, l \in \{4,5,\ldots, n\}$. 

From \eqref{a2} and \eqref{a3}, we have
\[[(1234), (1235)]=[(1234), (1325)]. \label{a8} \tag{a$_{8}$}\]
For $5 \le k < l \le n$, we apply $(5l), (4k)(5l) \in \mathfrak{S}_n$ 
to \eqref{a8} gives
\begin{align*}
[(1234), (123l)]&=[(1234), (132l)],\\
[(123k), (123l)]&=[(123k), (132l)].
\end{align*}
Considering \eqref{a8}, we see that 
\[[(1234), (1235)]=[(123k), (132l)]. \label{a9} \tag{a$_{9}$}\]
holds for all distinct $k, l \in \{4,5,\ldots, n\}$. 

Applying $(13) \in \mathfrak{S}_n$ to \eqref{a8} gives 
\[[(1234), (1235)]=[(3214), (3215)]=[(3214), (3125)]=[(1234), (1253)].\]
By a similar argument, we see that 
\[[(1234), (1235)]=[(123k), (12l3)] \label{a10} \tag{a$_{10}$}\]
holds for all distinct $k, l \in \{4,5,\ldots, n\}$. 

Applying $(23) \in \mathfrak{S}_n$ to \eqref{a9} gives 
\[[(1324), (1325)]=[(132k), (123l)],\]
where the right hand side is equal to $[(1234), (1235)]$ by \eqref{a9}. 
Hence 
\[[(1234), (1235)]=[(1324), (1325)].\] 
By a similar argument, we see that 
\[[(1234), (1235)]=[(132k), (132l)] \label{a11} \tag{a$_{11}$}\]
holds for all distinct $k, l \in \{4,5,\ldots, n\}$. 

Applying $(23) \in \mathfrak{S}_n$ to \eqref{a10} gives 
\[[(1324), (1325)]=[(132k), (13l2)]=[(132k), (12l3)].\]
Since the left hand side is equal to $[(1234), (1235)]$, we see that 
\[[(1234), (1235)]=[(132k), (12l3)] \label{a12} \tag{a$_{12}$}\]
holds for all distinct $k, l \in \{4,5,\ldots, n\}$. 

Applying $(13) \in \mathfrak{S}_n$ to \eqref{a11} gives 
 \[[(1234), (1235)]=
 [(3214), (3215)]= [(312k), (312l)]=[(12k3), (12l3)] \label{a13} \tag{a$_{13}$}\]
holds for all distinct $k, l \in \{4,5,\ldots, n\}$. 

The equalities \eqref{a7}, \eqref{a9}, \eqref{a10}, \eqref{a11}, \eqref{a12}, \eqref{a13} cover all the cases to be considered. 
\end{proof}

\begin{lem}\label{lem:comm4}
$[(ijkl), (stuv)]=1$ holds if $|\{i,j,k,l\} \cap \{s,t,u,v\}| =4$. 
\end{lem}
\begin{proof}
It suffices to show that $[(1234), (1243)]=1$ and $[(1234), (1324)]=1$ 
by considering the action of $\mathfrak{S}_n$. In fact, 
the latter follows from the former because 
$[(1234), (1324)]=[(4123), (4132)]$. 

When $\alpha^2=1$, we have
$\alpha [\alpha, \beta]=\beta \alpha \beta=[\beta, \alpha] \alpha$. 
Using this formula, we observe that $(1243)(1235)$ and $(1234)(1243)$ commutes. 
Indeed, from the pentagon relation for $\{1,2,4,3,5\}$, we have
\[(1243)(1235)=(1435)(1245)(2435).\]
Hence,
\begin{align*}
&(1243)(1235) \cdot (1234)(1243)\\
&=(1435)(1245)(2435) \cdot  (1234)(1243)\\
&=(1435)(1245)(1234)(2435)[(2435),(1234)](1243)\\
&=(1435)(1245)(1234)[(1234),(2435)](2435)(1243)\\
&=(1435)(1245)(1234) c_{234} (2435)(1243)\\
&=(1435)(1245)(1234) c_{234} [(2435),(1243)](1243)(2435)\\
&=(1435)(1245)(1234) c_{234}c_{234}(1243)(2435)\\
&=(1435)(1245)(1234)(1243)(2435)\\
&=(1435)(1234)(1245)c_{124}(1243)(2435)\\
&=(1435)(1234)c_{124}(1245)(1243)(2435)\\
&=(1435)(1234)c_{124}c_{124}(1243)(1245)(2435)\\
&=(1435)(1234)(1243)(1245)(2435)\\
&=(1234)(1435)c_{134}(1243)(1245)(2435)\\
&=(1234)c_{134}(1435)(1243)(1245)(2435)\\
&=(1234)c_{134}c_{134}(1243)(1435)(1245)(2435)\\
&=(1234)(1243)(1435)(1245)(2435)\\
&=(1234)(1243) \cdot (1243)(1235),
\end{align*}
\noindent
where we used the pentagon relation at the first and last equalities. 
Now we have 
\begin{align*}
(1234)(1243)&=(1243)(1243)(1235)(1235)(1234)(1243)\\
&=(1243)(1235)(1243)[(1243), (1235)](1235)(1234)(1243)\\
&=(1243)(1235)[(1235), (1243)](1243)(1235)(1234)(1243)\\
&=(1243)(1235) c_{123}(1243)(1235)(1234)(1243)\\
&=(1243)(1235)[(1235), (1234)](1243)(1235)(1234)(1243)\\
&=(1243)(1234)(1235)(1234) \underline{(1243)(1235)} \cdot \underline{(1234)(1243)}\\
&=(1243)(1234)(1235)(1234) \underline{(1234)(1243)} \cdot \underline{(1243)(1235)}\\
&=(1243)(1234). 
\end{align*}
\noindent
This completes the proof.
\end{proof}

\begin{lem}\label{lem:comm4center}
If $|\{i,j,k,l\} \cap \{s,t,u,v\}| =4$, the element $(ijkl)(stuv)$ is in the center of $\Gamma_n^4$. 
\end{lem}
\begin{proof}
It suffices to show that $(1234)(1243)$ is in the center. When 
$|\{1,2,3,4\} \cap \{i,j,k,l\}| \le 2$, the commutative relation says that 
$(1234)(1243)$ and $(ijkl)$ commutes. 
This holds also when 
$|\{1,2,3,4\} \cap \{i,j,k,l\}|=4$ because of Lemma \ref{lem:comm4}. 

The remaining case is when $|\{1,2,3,4\} \cap \{i,j,k,l\}|=3$. Say $\{1,2,3,4\} \cap \{i,j,k,l\}=\{i,j,k\}$. 
Then $c_{ijk}$ commutes with $(1243)$ and $(1234)$. We have  
\begin{align*}
(1234)(1243)(ijkl)&=(1234)(ijkl)(1243)c_{ijk}=(1234)(ijkl)c_{ijk}(1243)\\
&=(ijkl)(1234)c_{ijk}c_{ijk}(1243)=(ijkl)(1234)(1243).
\end{align*}
\noindent
This completes the proof.
\end{proof}

When $|\{i,j,k,l\} \cap \{s,t,u,v\}| =4$, we write 
\[\displaystyle\binom{ijkl}{stuv}:=(ijkl)(stuv)=(stuv)(ijkl),\]
which is in the center of $\Gamma_n^4$ as we have just seen. 

\begin{proof}[Proof of Proposition $\ref{prop:finitequotient}$]
By the $\mathfrak{S}_n$-action, it suffices to show that $c_{123}=c_{124}$. 
The pentagon relation for $\{3,4,1,2,5\}$ says that 
\[(3412)(3425)(4125)(3415)(3125)=1.\]
From this, we have 
\[(3415)(3125)(3412)=(4125)(3415).\]
We use this equality to see that 
\begin{align*}
c_{123}&=[(3412),(3125)]=(3412)(3125)(3412)(3125)\\
&=(3412)(3415)(3415)(3125)(3412)(3125)\\
&=(3412)(3415)(4125)(3415)(3125)\\
&=\binom{3412}{3421} \binom{4125}{4215} \binom{3125}{3215} 
(3421)(3415)(4215)(3425)(3215)\\
&=\binom{3412}{3421} \binom{4125}{4215} \binom{3125}{3215},
\end{align*}
\noindent
where the last equality follows from the pentagon relation for $\{3,4,2,1,5\}$. 
Hence 
\[c_{123}=[(1235),(1236)]=\binom{3412}{3421} \binom{4125}{4215} \binom{3125}{3215}.\]
Applying the transposition $(34) \in \mathfrak{S}_n$, we have
\begin{align*}
c_{124}&=[(1245),(1246)]=\binom{4312}{4321} \binom{3125}{3215} \binom{4125}{4215}\\
&=\binom{3412}{3421} \binom{3125}{3215} \binom{4125}{4215}=c_{123}.
\end{align*}
\noindent
This completes the proof.
\end{proof}

\section{The group $\Gamma_n^4$ is not abelian for $n \ge 6$}\label{sec:nonabelian}

We will show that $\Gamma_n^4=\overline{\Gamma_n^4}$ is not an abelian group for $n \ge 6$ 
by constructing a matrix representation 
whose restriction to $[\Gamma_n^4, \Gamma_n^4]$ is non-trivial. Specifically, we will prove that $c=[(1234), (1235)] \in \Gamma_n^4$ is a non-trivial element. 
The construction of the representation 
uses the presentation of $\Gamma_n^4$ in Definition \ref{def:gamma}.  For simplicity, we 
{\it identify} each $(ijkl) \in \mathcal{G}_n$ with its minimal expression $(i'j'k'l') \in \mathcal{G}_n^d$. 
This identification enables us to ignore the dihedral relations in a natural way. 

Recall the minimal generating set 
\[\Lambda_n =\{(1 2 3 l) \mid 4 \le l \le n\} \sqcup \{(1 j 2 l) \mid 3 \le j < l \le n\} \sqcup \{(1 j l k) \mid 2 \le j < k < l \le n\} \] 
of $\Gamma_n^4$ in Theorem \ref{thm:generating_gamma}. This set gives a basis of 
$H_1 (\Gamma_n^4) \cong (\mathbb{Z}/2\mathbb{Z})^{N_n}$. 
We endow the set $\Lambda_n$ with the following total order: 
Let $(s t u v), (s' t' u' v') \in \Lambda_n$. 
We define $(s t u v) \prec (s' t' u' v')$ if 
\begin{itemize}
\item[(I)] \ $\max \{s, t, u, v\} < \max \{s', t', u', v'\}$; or 
\item[(II)] \ $(s, t, u, v)$ is prior to $(s', t', u', v')$ in the usual lexicographic order when $\max \{s, t, u, v\} = \max \{s', t', u', v'\}$. 
\end{itemize}
\noindent
The sorted generators are denoted by 
\[X_1=(1234), \quad X_2=(1235), \quad \ldots \quad , \quad X_{N_n}=(1 (n-1) 2 n)\]
satisfying $X_1 \prec X_2 \prec \cdots \prec X_{N_n}$. 
Note that the rule (I) makes the order $\prec$ stabilizable, 
namely the order of $\Lambda_{n+1}$ is obtained 
from that of $\Lambda_n$ by putting the new generators $\Lambda_{n+1} - \Lambda_n$ 
behind $\Lambda_n$. 
For example, the total order of $\Lambda_6$ is given by 
\begin{align*}
& (1234) \prec (1243) \prec (1324) \\
\prec \ &(1235) \prec (1253) \prec (1254) \prec (1325) \prec (1354) \prec (1425)\\
\prec \ & (1236) \prec (1263) \prec (1264) \prec (1265) \prec (1326) \prec (1364)\\
\prec \ & (1365) \prec (1426) \prec (1465) \prec (1526).
\end{align*}
We identify $H_1 (\Gamma_n^4)$ with $(\mathbb{Z}/2\mathbb{Z})^{N_n}$ under the ordered basis $\Lambda_n$. 
We consider $(\mathbb{Z}/2\mathbb{Z})^{N_n}$ to be a vector space of column vectors.

From our computation in \cite[Section 2]{DTS} 
together with Proposition \ref{prop:finitequotient} and Lemma \ref{lem:comm4} in the previous section, 
we may write each of the elements of $\mathcal{G}_n^d$ 
by using the minimal generating set $\Lambda_n$ as follows.  

\begin{lem}\label{lem:bymingen}
The set $\mathcal{G}_n^d$ is divided into the following fifteen types, each of whose elements 
is written by using $\Lambda_n$ as in the right hand side 
(the underlined part will be explained later): 
\begin{itemize}
\item[(G1)] $\underline{(1 2 3 l)}$ with $4 \le l \le n$, 

\item[(G2)] $\underline{(1 j 2 l)}$ with $3 \le j < l \le n$, 

\item[(G3)] $\underline{(1 j l k)}$ with $2 \le j < k < l \le n$, 

\item[(A1)] 
$(1 2 k l)$ with $4 \le k < l \le n$,

$(1 2 k l)= c (123k)(12k3)(123l)(12l3)\underline{(12lk)}$, 

\item[(A2)] 
$(1 j k l)$ with $3 \le j < k < l \le n$,

$(1 j k l)= c (12kj)(1j2k) (12lj)(1j2l)\underline{(1jlk)}$,

\item[(A3)] 
$(2 j k l)$ with $3 \le j < k < l \le n$,

$(2 j k l)=\underline{(1j2k)(12lk)(1j2l)(1jlk)}$,

\item[(A4)] 
$(i j k l)$ with $3 \le i < j < k < l \le n$,

$(i j k l)=c (12kj)\underline{(1ikj)}(1j2k)(12lj)\underline{(1ilj)(1ilk)}(1j2l)\underline{(1jlk)}$, 

\item[(B1)] 
$(2 3 l k)$ with $4 \le k < l \le n$,

$(2 3 l k)=\underline{(123k)(123l)(12lk)(13lk)}$,

\item[(B2)] 
$(2 j l k)$ with $4 \le j < k < l \le n$,

$(2 j l k)=(123k)(12k3)\underline{(12kj)}(123l)(12l3)\underline{(12lj)(12lk)(1jlk)}$,

\item[(B3)] 
$(i j l k)$ with $3 \le i  < j < k < l \le n$, 

$(i j l k)=(12ki)(1i2k)\underline{(1ikj)}(12li)(1i2l)\underline{(1ilj)(1ilk)(1jlk)}$, 

\item[(C1)] 
$(1 k 3 l)$ with $4 \le  k < l \le n$,

$(1 k 3 l)=(123k)(132k)(12lk)\underline{(13lk)}(1k2l)$, 

\item[(C2)] 
$(1 k j l)$ with $4 \le j < k < l \le n$,

$(1 k j l)= c(123j)(12j3)(123k)(12k3)(12kj)(1j2k)(12lk)\underline{(1jlk)}(1k2l)$,

\item[(C3)] 
$(2 k j l)$ with $3 \le j < k < l \le n$, 

$(2 k j l)= c \underline{(1j2k)(12lj)(1jlk)(1k2l)}$, 

\item[(C4)] 
$(3 k j l)$ with $4 \le j < k < l \le n$, 

$(3 k j l)=(123j)(132j)(123k)(12kj)(132k)\underline{(13kj)}(1j2k)(12lk)\underline{(13lj)(13lk)}\\
\hskip 60pt \cdot \underline{(1jlk)}(1k2l)$,

\item[(C5)] 
$(i k j l)$ with $4 \le i < j < k < l\le n$,

$(i k j l)=(123j)(12j3)(12ji)(1i2j)(123k)(12k3)(12ki)(12kj)(1i2k)\underline{(1ikj)}\\
\hskip 60pt \cdot (1j2k)(12lk)\underline{(1ilj)(1ilk)(1jlk)}(1k2l)$.
\end{itemize}
\end{lem}
\noindent
As we saw in Theorem \ref{thm:generating_gamma}, 
the elements in (G1), (G2), (G3) form $\Lambda_n$. 
The total number of the above elements is $\displaystyle\frac{n(n-1)(n-2)(n-3)}{8}$, 
which coincides with the number of ordered quadruples of distinct integers in $[n]$ up to the dihedral action. 
 
Now we construct a matrix representation
\[r \colon \Gamma_n^4 \longrightarrow \operatorname{GL}_{N_n+2}(\mathbb{Z}/2\mathbb{Z})\]
of dimension $N_n+2 = \binom{n}{3}+1$. The construction is based on the well-known standard representation of the integral Heisenberg group. 
Our representation $r$ is decomposed as 
\[r (x) := \begin{pmatrix}
1 & \mathfrak{b}(x) & \mathfrak{e}(x) \\
0_{N_n} & I_{N_n} & \mathfrak{a} (x) \\
0 & {}^t 0_{N_n} & 1
\end{pmatrix}\]
for $x \in \Gamma_n^4$, 
where $\mathfrak{b}(x)$ is a row vector of size $N_n$, $0_{N_n}$ and $\mathfrak{a}(x)$ are column vectors of size $N_n$ 
and $\mathfrak{e} (x) \in \mathbb{Z}/2\mathbb{Z}$. Here, ${}^t $ denotes the transpose. 
For $r$ to be a representation, 
$\mathfrak{a}$ and $\mathfrak{b}$ should be homomorphisms and 
$\mathfrak{e}$ should satisfy $\mathfrak{e} (xy)=\mathfrak{e} (x)+\mathfrak{e} (y)+\mathfrak{b}(x) \mathfrak{a} (y)$ 
for any $x,y \in \Gamma_n^4$. 

The column vector $\mathfrak{a} (x)$ is given by the abelianization map 
\[\mathfrak{a} \colon \Gamma_n^4 \longrightarrow H_1 (\Gamma_n^4) \cong (\mathbb{Z}/2\mathbb{Z})^{N_n}, \]
so that the map $\mathfrak{a}$ is a homomorphism. 
The map $\mathfrak{b}$, which is also a homomorphism, is given by 
\[\mathfrak{b}(x)= {}^t(D \cdot \mathfrak{a}(x)),\]
where $D=(\bm{d}_1, \bm{d}_2, \ldots, \bm{d}_{N_n})$ 
is a square matrix of size $N_n$. The vector $\bm{d}_k$ is defined as 
\[\bm{d}_k := \sum_{j=1}^{k-1} \mu_k^j \bm{e}_j\]
where $\bm{e}_j$ is the $j$-th elementary column vector in $(\mathbb{Z}/2\mathbb{Z})^{N_n}$ and 
\[\mu_k^j=
\begin{cases}
1 & \text{if $X_j=(stuv)$ and $X_k=(s't'u'v')$ satisfy $|\{s,t,u,v\} \cap \{s', t', u',v'\}|=3$}\\
0 & \text{otherwise}
\end{cases}.\]
Consequently, 
\[x \circ y :=\mathfrak{b}(x) \mathfrak{a} (y) = {}^t(D \cdot \mathfrak{a}(x)) \mathfrak{a}(y)\qquad (x, y \in \Gamma_n^4)\]
induces a $(\mathbb{Z}/2\mathbb{Z})$-valued bilinear form over the $(\mathbb{Z}/2\mathbb{Z})$-vector space $H_1 (\Gamma_n^4)$. 

\begin{remark}
The above construction comes from the well-known normal form of $2$-step nilpotent $2$-groups.  
We refer to Sims' book \cite{Sims} for structures of nilpotent $p$-groups. 
From our discussion in Section \ref{sec:finite}, an element in $\Gamma_n^4$ is uniquely written of the form (called the normal form) 
\[c^{\varepsilon_0} X_1^{\varepsilon_1} X_2^{\varepsilon_2} X_3^{\varepsilon_3} \cdots X_{N_n}^{\varepsilon_{N_n}} 
\qquad (\varepsilon_i \in \{0,1\})\]
{\it up to} $c$, which has not yet been shown  
to be non-trivial, but is in the center and satisfies $c^2=1$. 
Note that the right hand sides of equalities in Lemma \ref{lem:bymingen} are all normal forms.  
Given two elements $x, y \in \Gamma_n^4$ having normal forms 
$x=c^{\varepsilon_0} X_1^{\varepsilon_1} X_2^{\varepsilon_2} \cdots X_{N_n}^{\varepsilon_{N_n}}$ and 
$y=c^{\eta_0} X_1^{\eta_1} X_2^{\eta_2}  \cdots X_{N_n}^{\eta_{N_n}}$, 
the normal form of their product is obtained by applying $X_k X_j =c^{\mu_{jk}}X_jX_k$ in an appropriate order 
for all possible $j < k$ with $\varepsilon_k \eta_j=1$. 
The result is 
\[c^{\varepsilon_0+\eta_0+x \circ y}
X_1^{\varepsilon_1+\eta_1}X_2^{\varepsilon_2+\eta_2}  \cdots X_{N_n}^{\varepsilon_{N_n}+\eta_{N_n}}.\]
\end{remark}

We see from the above description that the matrix $r(x)$ is  recovered from 
\[(\mathfrak{e}(x), \mathfrak{a}(x)) \in (\mathbb{Z}/2\mathbb{Z}) \times H_1 (\Gamma_n^4).\] 
This gives a description of the central extension in Theorem \ref{thm:main}.  
We denote this central extension by $(\mathbb{Z}/2\mathbb{Z}) \widetilde{\times} H_1 (\Gamma_n^4)$. Indeed, 
the product rule of $(\mathbb{Z}/2\mathbb{Z}) \widetilde{\times} H_1 (\Gamma_n^4)$ obtained from 
products of matrices is
\[(\mathfrak{e}(x), \mathfrak{a}(x))  (\mathfrak{e}(y), \mathfrak{a}(y)) =
(\mathfrak{e}(x)+\mathfrak{e}(y)+ x \circ y, \mathfrak{a}(x) + \mathfrak{a}(y)).\]
We regard $(\mathbb{Z}/2\mathbb{Z}) \widetilde{\times} H_1 (\Gamma_n^4)$ as a subgroup of $\operatorname{GL}_{N_n+2}(\mathbb{Z}/2\mathbb{Z})$. 
The element $(0,0) \in (\mathbb{Z}/2\mathbb{Z}) \widetilde{\times} H_1 (\Gamma_n^4)$ 
corresponds to $I_{N_n+2} \in \operatorname{GL}_{N_n+2}(\mathbb{Z}/2\mathbb{Z})$. 

Finally, we put 
\[
\mathfrak{e}(X)=\begin{cases}
1 & \text{if $X$ is of type (A1), (A2), (A4), (C2) or (C3),}\\
0 & \text{otherwise}
\end{cases}\]
for $X \in \mathcal{G}_n^d$. We now prove the following. 
\begin{prop}\label{prop:representation}
Let $n \ge 4$. The above assignment to the generating set $\mathcal{G}_n^d$ defines a matrix representation 
\[r \colon \Gamma_n^4 \longrightarrow (\mathbb{Z}/2\mathbb{Z}) \widetilde{\times} H_1 (\Gamma_n^4) 
\subset \operatorname{GL}_{N_n+2} (\mathbb{Z}/2\mathbb{Z})\]
satisfying $r( c )= (1, 0) \neq (0,0)$. Therefore the commutator subgroup 
$[\Gamma_n^4, \Gamma_n^4]$ is non-trivial. 
\end{prop}

Our proof of Proposition \ref{prop:representation} is straightforward. It is given by 
checking that the assignment $\mathfrak{e}(X)$ to $X \in \mathcal{G}_n^d$ preserves the relations 
in the original presentation of $\Gamma_n^4$. 
We may ignore the dihedral relations as mentioned in the beginning of this section. 
We will write all the data for our proof in the rest of this section. 
One can also give a computer aided proof, which is discussed in the next section. 

\begin{lem}\label{lem:relinvol}
For any generator $X \in \mathcal{G}_n^d$, the equalities $r(X)^2 =(0,0)$ and $X \circ X=0$ hold.
\end{lem}
\begin{proof}
Since $r(X)^2=(\mathfrak{e}(X), \mathfrak{a} (X))^2=(X \circ X, 0)$, we only need to consider $X \circ X$. 
To compute $X \circ X$, we write $X= x_1 x_2 \cdots x_k$ where $x_i \in \Lambda_n$ modulo $c$. 
It follows from the definition of the vector $\mathfrak{b}$ 
that the equality $x_j \circ x_j =0$ holds for $x_j \in \Lambda_n$. Hence we have 
\[X \circ X = \sum_{1 \le i < j \le k} (x_i \circ x_j + x_j \circ x_i).\]
Again, the definition of the vector $\mathfrak{b}$ says that 
\begin{align*}
& x_i \circ x_j + x_j \circ x_i \\
&= \begin{cases}
1 & \text{if $x_i=(stuv)$ and $x_j=(s't'u'v')$ satisfy $|\{s,t,u,v\} \cap \{s', t', u',v'\}|=3$}\\
0 & \text{otherwise}
\end{cases}.
\end{align*}
Therefore $X \circ X$ counts the number of the pairs $\{x_i, x_j\}$ where $x_i$ and $x_j$ share just $3$ letters. 
Note that for $x_i=(stuv)$ and $x_j=(s't'u'v')$, the number $x_i \circ x_j + x_j \circ x_i \in \mathbb{Z}/2\mathbb{Z}$ 
depends only on the {\it sets} $\{s,t,u,v\}$ and $\{s',t',u',v'\}$. So, for example, 
in the case of 
\[X=(1245)=c (1234) (1243) (1235)(1253) (1254)\]
of type (A1), the equality
$(1234) \circ x_l+x_l \circ (1234)=(1243) \circ x_l+x_l \circ (1243)$ holds for all $x_l \in \Lambda_n$. So 
$(1234) \circ x_l+x_l \circ (1234) + (1243) \circ x_l+x_l \circ (1243)$
contributes $0$ or $2 \equiv 0$ to $X \circ X$. Hence in our counting process, we may ignore such pairs. 
After removing them, the underlined parts remain in the right hand sides of the equalities of Lemma \ref{lem:bymingen}. 
They are divided into two patterns. Some have the word length $1$ and the others have $4$. 
In the latter case, the underlined part consists of four elements of $\Lambda_n$ obtained from the quadruples 
\[\{1,i,j,k\}, \quad \{1,i,j,l\}, \quad \{1, i, k, l\}, \quad \{1,j, k, l\}\]
for some $\{i,j,k,l\} \subset \{2,3,\ldots,n\}$. 
It is easy to see that $X \circ X=0$ holds by just looking these underlined parts. 
\end{proof}

\begin{lem}\label{lem:relcomm}
For any pair of generators $X=(stuv), Y=(s't'u'v') \in \mathcal{G}_n^d$ satisfying 
$|\{s,t,u,v\} \cap \{s', t', u',v'\}| \le 2$, the equality $[r(X),r(Y)] =(0,0)$ holds.
\end{lem}
\begin{proof}
Using Lemma \ref{lem:relinvol}, we have 
\[[r(X),r(Y)]=((\mathfrak{e}(X), \mathfrak{a} (X))(\mathfrak{e}(Y), \mathfrak{a} (Y)))^2
=(X \circ Y + Y \circ X, 0).\]
Write $X= x_1 x_2 \cdots x_k$ and $Y= x_{k+1} x_{k+2} \cdots x_{k+l}$ 
where $x_i \in \Lambda_n$. Then 
\[X \circ Y + Y \circ X=\sum_{\begin{subarray}{c}
1\le i \le k \\ 1 \le j \le l\end{subarray}} (x_i \circ x_{k+j}+x_{k+j} \circ x_i).\]
Then by an argument similar to the one in the previous lemma, 
we may just look the underlined parts. It is easy to check that $X \circ Y + Y \circ X=0$ holds. 
\end{proof}

\begin{lem}\label{lem:relpentagon}
For any distinct $s,t,u,v,w \in [n]$, the equality 
\[r((stuv))r((stvw))r((tuvw))r((stuw))r((suvw)) =(0,0)\]
holds.
\end{lem}
\begin{proof}
Put 
\[x_1=(stuv), \quad x_2=(stvw), \quad x_3=(tuvw), \quad x_4=(stuw), \quad x_5=(suvw)\]
for the pentagon relation for $\{s,t,u,v,w\}$. We have
\begin{align*}
r(x_1)r(x_2)r(x_3)r(x_4)r(x_5)&=\left( \sum_{i=1}^5 \mathfrak{e}(x_i) + 
\sum_{j=1}^4 (x_1 x_2 \cdots x_j) \circ x_{j+1}, 0 \right)\\
&=\left( \sum_{i=1}^5 \mathfrak{e}(x_i) +x_1 \circ x_2 + x_1 \circ x_3 +x_2 \circ x_3 + x_5 \circ x_4, 0 \right),
\end{align*}
\noindent
where at the second equality we used the equalities 
\[(x_1x_2 x_3) \circ x_4 = (x_5 x_4) \circ x_4 = x_5 \circ x_4, \qquad 
(x_1x_2 x_3 x_4) \circ x_5 =x_5 \circ x_5=0.\]
\noindent
We will check that 
\[\sum_{i=1}^5 \mathfrak{e}(x_i) +x_1 \circ x_2 + x_1 \circ x_3 +x_2 \circ x_3 + x_5 \circ x_4=0\]
holds. 

As remarked in \cite[Section 5.2]{DTS}, the pentagon relations have a dihedral symmetry. 
Indeed, the pentagon relation for $\{s,t,u,v,w\}$ together with the involutive and dihedral relations 
derives that for $\{t,u,v,w,s\}$ and for $\{w,v,u,t,s\}$. Therefore, for the proof of the lemma, 
it suffices to check the cases corresponding to the pentagon relations 
for $\{s,t,u,v,w\}$ satisfying $s=\min \{s,t,u,v,w\}$ and $t < w$. 

Assume that the integers $i, j, k, l, m$ satisfy $i < j < k < l < m$. 
The following twelve cases cover all we should consider. In the tables below, the middle block writes the type and 
the value of $\mathfrak{e}$ for each of $x_1, x_2, \ldots, x_5$. 
We see that for all cases, the sum of all numbers in a row is equal to $0$ in $\mathbb{Z}/2\mathbb{Z}$. 

(1) Pentagon relation for $\{i, j, k, l, m\}$: 
\[\begin{array}{|l|ccccc|cccc|}
\hline
& \begin{matrix} x_1 \\ (ijkl) \end{matrix} & 
\begin{matrix} x_2 \\ (ijlm) \end{matrix} & 
\begin{matrix} x_3 \\ (jklm) \end{matrix} & 
\begin{matrix} x_4 \\ (ijkm) \end{matrix} & 
\begin{matrix} x_5 \\ (iklm) \end{matrix} & x_1 \circ x_2 & x_1 \circ x_3 & x_2 \circ x_3 & x_5 \circ x_4\\\hline
\begin{smallmatrix} i=1 \\ j=2 \\ k=3 \end{smallmatrix} & 
\begin{matrix} \text{(G1)} \\ 0 \end{matrix} &  \begin{matrix} \text{(A1)} \\ 1 \end{matrix} & 
\begin{matrix} \text{(A3)} \\ 0 \end{matrix} &  \begin{matrix} \text{(G1)} \\ 0 \end{matrix} &
\begin{matrix} \text{(A2)} \\ 1 \end{matrix} & 0 & 0 & 1 & 1 \\ \hline
\begin{smallmatrix} i=1 \\ j=2 \\ k\ge 4 \end{smallmatrix} & 
\begin{matrix} \text{(A1)} \\ 1 \end{matrix} &  \begin{matrix} \text{(A1)} \\ 1 \end{matrix} & 
\begin{matrix} \text{(A3)} \\ 0 \end{matrix} &  \begin{matrix} \text{(A1)} \\ 1 \end{matrix} &
\begin{matrix} \text{(A2)} \\ 1 \end{matrix} & 0 & 0 & 1 & 1\\ \hline
\begin{smallmatrix} i=1 \\ j \ge 3 \end{smallmatrix} & 
\begin{matrix} \text{(A2)} \\ 1 \end{matrix} &  \begin{matrix} \text{(A2)} \\ 1 \end{matrix} & 
\begin{matrix} \text{(A4)} \\ 1 \end{matrix} &  \begin{matrix} \text{(A2)} \\ 1 \end{matrix} &
\begin{matrix} \text{(A2)} \\ 1 \end{matrix} & 0 & 0 & 0 & 1\\ \hline
\begin{smallmatrix}i=2 \end{smallmatrix}& 
\begin{matrix} \text{(A3)} \\ 0 \end{matrix} &  \begin{matrix} \text{(A3)} \\ 0 \end{matrix} & 
\begin{matrix} \text{(A4)} \\ 1 \end{matrix} &  \begin{matrix} \text{(A3)} \\ 0 \end{matrix} &
\begin{matrix} \text{(A3)} \\ 0 \end{matrix} & 1& 0 & 1 & 1\\ \hline
\begin{smallmatrix}i \ge 3 \end{smallmatrix}& 
\begin{matrix} \text{(A4)} \\ 1 \end{matrix} &  \begin{matrix} \text{(A4)} \\ 1 \end{matrix} & 
\begin{matrix} \text{(A4)} \\ 1 \end{matrix} &  \begin{matrix} \text{(A4)} \\ 1 \end{matrix} &
\begin{matrix} \text{(A4)} \\ 1 \end{matrix} & 0 & 1 & 0 & 0\\ \hline
\end{array}
\] 

(2) Pentagon relation for $\{i, j, k, m, l\}$: 
\[\begin{array}{|l|ccccc|cccc|}
\hline
& \begin{matrix} x_1 \\ (ijkm) \end{matrix} & 
\begin{matrix} x_2 \\ (ijml) \end{matrix} & 
\begin{matrix} x_3 \\ (jkml) \end{matrix} & 
\begin{matrix} x_4 \\ (ijkl) \end{matrix} & 
\begin{matrix} x_5 \\ (ikml) \end{matrix} & x_1 \circ x_2 & x_1 \circ x_3 & x_2 \circ x_3 & x_5 \circ x_4\\\hline
\begin{smallmatrix} i=1 \\ j=2 \\ k=3 \end{smallmatrix} & 
\begin{matrix} \text{(G1)} \\ 0 \end{matrix} &  \begin{matrix} \text{(G3)} \\ 0 \end{matrix} & 
\begin{matrix} \text{(B1)} \\ 0 \end{matrix} &  \begin{matrix} \text{(G1)} \\ 0 \end{matrix} &
\begin{matrix} \text{(G3)} \\ 0 \end{matrix} & 0 & 1 & 0 & 1\\ \hline
\begin{smallmatrix} i=1 \\ j=2 \\ k\ge 4 \end{smallmatrix} & 
\begin{matrix} \text{(A1)} \\ 1 \end{matrix} &  \begin{matrix} \text{(G3)} \\ 0 \end{matrix} & 
\begin{matrix} \text{(B2)} \\ 0 \end{matrix} &  \begin{matrix} \text{(A1)} \\ 1 \end{matrix} &
\begin{matrix} \text{(G3)} \\ 0 \end{matrix} & 0 & 1 & 0 & 1\\ \hline
\begin{smallmatrix} i=1 \\ j \ge 3 \end{smallmatrix} & 
\begin{matrix} \text{(A2)} \\ 1 \end{matrix} &  \begin{matrix} \text{(G3)} \\ 0 \end{matrix} & 
\begin{matrix} \text{(B3)} \\ 0 \end{matrix} &  \begin{matrix} \text{(A2)} \\ 1 \end{matrix} &
\begin{matrix} \text{(G3)} \\ 0 \end{matrix} & 0 & 1 & 0 & 1\\ \hline
\begin{smallmatrix}i=2 \\ j=3 \end{smallmatrix}& 
\begin{matrix} \text{(A3)} \\ 0 \end{matrix} &  \begin{matrix} \text{(B1)} \\ 0 \end{matrix} & 
\begin{matrix} \text{(B3)} \\ 0 \end{matrix} &  \begin{matrix} \text{(A3)} \\ 0 \end{matrix} &
\begin{matrix} \text{(B2)} \\ 0 \end{matrix} & 0 & 0 & 1 & 1\\ \hline
\begin{smallmatrix}i=2 \\ j \ge 4 \end{smallmatrix}& 
\begin{matrix} \text{(A3)} \\ 0 \end{matrix} &  \begin{matrix} \text{(B2)} \\ 0 \end{matrix} & 
\begin{matrix} \text{(B3)} \\ 0 \end{matrix} &  \begin{matrix} \text{(A3)} \\ 0 \end{matrix} &
\begin{matrix} \text{(B2)} \\ 0 \end{matrix} & 0 & 0 & 1 & 1\\ \hline
\begin{smallmatrix}i \ge 3 \end{smallmatrix}& 
\begin{matrix} \text{(A4)} \\ 1 \end{matrix} &  \begin{matrix} \text{(B3)} \\ 0 \end{matrix} & 
\begin{matrix} \text{(B3)} \\ 0 \end{matrix} &  \begin{matrix} \text{(A4)} \\ 1 \end{matrix} &
\begin{matrix} \text{(B3)} \\ 0 \end{matrix} & 1 & 0 & 0 & 1\\ \hline
\end{array}
\]

(3) Pentagon relation for $\{i, j, l, k, m\}$: 
\[\begin{array}{|l|ccccc|cccc|}
\hline
& \begin{matrix} x_1 \\ (ijlk) \end{matrix} & 
\begin{matrix} x_2 \\ (ijkm) \end{matrix} & 
\begin{matrix} x_3 \\ (jlkm) \end{matrix} & 
\begin{matrix} x_4 \\ (ijlm) \end{matrix} & 
\begin{matrix} x_5 \\ (ilkm) \end{matrix} & x_1 \circ x_2 & x_1 \circ x_3 & x_2 \circ x_3 & x_5 \circ x_4\\\hline
\begin{smallmatrix} i=1 \\ j=2 \\ k=3 \end{smallmatrix} & 
\begin{matrix} \text{(G3)} \\ 0 \end{matrix} &  \begin{matrix} \text{(G1)} \\ 0 \end{matrix} & 
\begin{matrix} \text{(C3)} \\ 1 \end{matrix} &  \begin{matrix} \text{(A1)} \\ 1 \end{matrix} &
\begin{matrix} \text{(C1)} \\ 0 \end{matrix} & 0 & 0 & 1 & 1\\ \hline
\begin{smallmatrix} i=1 \\ j=2 \\ k\ge 4 \end{smallmatrix} & 
\begin{matrix} \text{(G3)} \\ 0 \end{matrix} &  \begin{matrix} \text{(A1)} \\ 1 \end{matrix} & 
\begin{matrix} \text{(C3)} \\ 1 \end{matrix} &  \begin{matrix} \text{(A1)} \\ 1 \end{matrix} &
\begin{matrix} \text{(C2)} \\ 1 \end{matrix} & 0 & 0 & 1 & 1\\ \hline
\begin{smallmatrix} i=1 \\ j = 3 \end{smallmatrix} & 
\begin{matrix} \text{(G3)} \\ 0 \end{matrix} &  \begin{matrix} \text{(A2)} \\ 1 \end{matrix} & 
\begin{matrix} \text{(C4)} \\ 0 \end{matrix} &  \begin{matrix} \text{(A2)} \\ 1 \end{matrix} &
\begin{matrix} \text{(C2)} \\ 1 \end{matrix} & 0 & 1 & 0 & 0\\ \hline
\begin{smallmatrix}i=1 \\ j \ge 4 \end{smallmatrix}& 
\begin{matrix} \text{(G3)} \\ 0 \end{matrix} &  \begin{matrix} \text{(A2)} \\ 1 \end{matrix} & 
\begin{matrix} \text{(C5)} \\ 0 \end{matrix} &  \begin{matrix} \text{(A2)} \\1 \end{matrix} &
\begin{matrix} \text{(C2)} \\ 1 \end{matrix} & 0 & 1 & 0 & 0\\ \hline
\begin{smallmatrix}i=2 \\ j =3 \end{smallmatrix}& 
\begin{matrix} \text{(B1)} \\ 0 \end{matrix} &  \begin{matrix} \text{(A3)} \\ 0 \end{matrix} & 
\begin{matrix} \text{(C4)} \\ 0 \end{matrix} &  \begin{matrix} \text{(A3)} \\ 0 \end{matrix} &
\begin{matrix} \text{(C3)} \\ 1 \end{matrix} & 1 & 0 & 0 & 0\\ \hline
\begin{smallmatrix}i=2 \\ j \ge 4 \end{smallmatrix}& 
\begin{matrix} \text{(B2)} \\ 0 \end{matrix} &  \begin{matrix} \text{(A3)} \\ 0 \end{matrix} & 
\begin{matrix} \text{(C5)} \\ 0 \end{matrix} &  \begin{matrix} \text{(A3)} \\ 0 \end{matrix} &
\begin{matrix} \text{(C3)} \\ 1 \end{matrix} & 1 & 0 & 0 & 0\\ \hline
\begin{smallmatrix}i = 3 \end{smallmatrix}& 
\begin{matrix} \text{(B3)} \\ 0 \end{matrix} &  \begin{matrix} \text{(A4)} \\ 1 \end{matrix} & 
\begin{matrix} \text{(C5)} \\ 0 \end{matrix} &  \begin{matrix} \text{(A4)} \\ 1 \end{matrix} &
\begin{matrix} \text{(C4)} \\ 0 \end{matrix} & 1 & 1 & 1 & 1\\ \hline
\begin{smallmatrix}i \ge 4 \end{smallmatrix}& 
\begin{matrix} \text{(B3)} \\ 0 \end{matrix} &  \begin{matrix} \text{(A4)} \\ 1 \end{matrix} & 
\begin{matrix} \text{(C5)} \\ 0 \end{matrix} &  \begin{matrix} \text{(A4)} \\ 1 \end{matrix} &
\begin{matrix} \text{(C5)} \\ 0 \end{matrix} & 1 & 1 & 1 & 1\\ \hline
\end{array}
\] 

(4) Pentagon relation for $\{i, j, l, m, k\}$: 
\[\begin{array}{|l|ccccc|cccc|}
\hline
& \begin{matrix} x_1 \\ (ijlm) \end{matrix} & 
\begin{matrix} x_2 \\ (ijmk) \end{matrix} & 
\begin{matrix} x_3 \\ (jkml) \end{matrix} & 
\begin{matrix} x_4 \\ (ijlk) \end{matrix} & 
\begin{matrix} x_5 \\ (ikml) \end{matrix} & x_1 \circ x_2 & x_1 \circ x_3 & x_2 \circ x_3 & x_5 \circ x_4\\\hline
\begin{smallmatrix} i=1 \\ j=2 \\ k=3 \end{smallmatrix} & 
\begin{matrix} \text{(A1)} \\ 1 \end{matrix} &  \begin{matrix} \text{(G3)} \\ 0 \end{matrix} & 
\begin{matrix} \text{(B1)} \\ 0 \end{matrix} &  \begin{matrix} \text{(G3)} \\ 0 \end{matrix} &
\begin{matrix} \text{(G3)} \\ 0 \end{matrix} & 1 & 0 & 1 & 1\\ \hline
\begin{smallmatrix} i=1 \\ j=2 \\ k\ge 4 \end{smallmatrix} & 
\begin{matrix} \text{(A1)} \\ 1 \end{matrix} &  \begin{matrix} \text{(G3)} \\ 0 \end{matrix} & 
\begin{matrix} \text{(B2)} \\ 0  \end{matrix} &  \begin{matrix} \text{(G3)} \\ 0 \end{matrix} &
\begin{matrix} \text{(G3)} \\ 0 \end{matrix} & 1 & 0 & 1 & 1\\ \hline
\begin{smallmatrix} i=1 \\ j \ge 3\end{smallmatrix} & 
\begin{matrix} \text{(A2)} \\ 1 \end{matrix} &  \begin{matrix} \text{(G3)} \\ 0 \end{matrix} & 
\begin{matrix} \text{(B3)} \\ 0 \end{matrix} &  \begin{matrix} \text{(G3)} \\ 0 \end{matrix} &
\begin{matrix} \text{(G3)} \\ 0 \end{matrix} & 1 & 0 & 1 & 1\\ \hline
\begin{smallmatrix}i=2 \\ j =3 \end{smallmatrix}& 
\begin{matrix} \text{(A3)} \\ 0 \end{matrix} &  \begin{matrix} \text{(B1)} \\ 0 \end{matrix} & 
\begin{matrix} \text{(B3)} \\ 0 \end{matrix} &  \begin{matrix} \text{(B1)} \\ 0 \end{matrix} &
\begin{matrix} \text{(B2)} \\ 0 \end{matrix} & 1 & 1 & 0 & 0\\ \hline
\begin{smallmatrix}i=2 \\ j \ge 4 \end{smallmatrix}& 
\begin{matrix} \text{(A3)} \\ 0 \end{matrix} &  \begin{matrix} \text{(B2)} \\ 0 \end{matrix} & 
\begin{matrix} \text{(B3)} \\ 0 \end{matrix} &  \begin{matrix} \text{(B2)} \\ 0 \end{matrix} &
\begin{matrix} \text{(B2)} \\ 0 \end{matrix} & 1 & 1 & 0 & 0\\ \hline
\begin{smallmatrix}i \ge  3 \end{smallmatrix}& 
\begin{matrix} \text{(A4)} \\ 1 \end{matrix} &  \begin{matrix} \text{(B3)} \\ 0 \end{matrix} & 
\begin{matrix} \text{(B3)} \\ 0 \end{matrix} &  \begin{matrix} \text{(B3)} \\ 0 \end{matrix} &
\begin{matrix} \text{(B3)} \\ 0 \end{matrix} & 0 & 0 & 1 & 0\\ \hline
\end{array}
\] 

(5) Pentagon relation for $\{i, j, m, k, l\}$: 
\[\begin{array}{|l|ccccc|cccc|}
\hline
& \begin{matrix} x_1 \\ (ijmk) \end{matrix} & 
\begin{matrix} x_2 \\ (ijkl) \end{matrix} & 
\begin{matrix} x_3 \\ (jlkm) \end{matrix} & 
\begin{matrix} x_4 \\ (ijml) \end{matrix} & 
\begin{matrix} x_5 \\ (ilkm) \end{matrix} & x_1 \circ x_2 & x_1 \circ x_3 & x_2 \circ x_3 & x_5 \circ x_4\\\hline
\begin{smallmatrix} i=1 \\ j=2 \\ k=3 \end{smallmatrix} & 
\begin{matrix} \text{(G3)} \\ 0 \end{matrix} &  \begin{matrix} \text{(G1)} \\ 0 \end{matrix} & 
\begin{matrix} \text{(C3)} \\ 1 \end{matrix} &  \begin{matrix} \text{(G3)} \\ 0 \end{matrix} &
\begin{matrix} \text{(C1)} \\ 0 \end{matrix} & 1 & 1 & 0 & 1 \\ \hline
\begin{smallmatrix} i=1 \\ j=2 \\ k\ge 4 \end{smallmatrix} & 
\begin{matrix} \text{(G3)} \\ 0 \end{matrix} &  \begin{matrix} \text{(A1)} \\ 1 \end{matrix} & 
\begin{matrix} \text{(C3)} \\ 1 \end{matrix} &  \begin{matrix} \text{(G3)} \\ 0 \end{matrix} &
\begin{matrix} \text{(C2)} \\ 1 \end{matrix} & 1 & 1 & 0 & 1\\ \hline
\begin{smallmatrix} i=1 \\ j = 3 \end{smallmatrix} & 
\begin{matrix} \text{(G3)} \\ 0 \end{matrix} &  \begin{matrix} \text{(A2)} \\ 1 \end{matrix} & 
\begin{matrix} \text{(C4)} \\ 0 \end{matrix} &  \begin{matrix} \text{(G3)} \\ 0 \end{matrix} &
\begin{matrix} \text{(C2)} \\ 1 \end{matrix} & 1 & 1 & 0 & 0\\ \hline
\begin{smallmatrix}i=1 \\ j \ge 4 \end{smallmatrix}& 
\begin{matrix} \text{(G3)} \\ 0 \end{matrix} &  \begin{matrix} \text{(A2)} \\ 1 \end{matrix} & 
\begin{matrix} \text{(C5)} \\ 0 \end{matrix} &  \begin{matrix} \text{(G3)} \\ 0 \end{matrix} &
\begin{matrix} \text{(C2)} \\ 1 \end{matrix} & 1 & 1 & 0 & 0\\ \hline
\begin{smallmatrix}i=2 \\ j =3 \end{smallmatrix}& 
\begin{matrix} \text{(B1)} \\ 0 \end{matrix} &  \begin{matrix} \text{(A3)} \\ 0 \end{matrix} & 
\begin{matrix} \text{(C4)} \\ 0 \end{matrix} &  \begin{matrix} \text{(B1)} \\ 0 \end{matrix} &
\begin{matrix} \text{(C3)} \\ 1 \end{matrix} & 0 & 1 & 1 & 1\\ \hline
\begin{smallmatrix}i=2 \\ j \ge 4 \end{smallmatrix}& 
\begin{matrix} \text{(B2)} \\ 0 \end{matrix} &  \begin{matrix} \text{(A3)} \\ 0 \end{matrix} & 
\begin{matrix} \text{(C5)} \\ 0 \end{matrix} &  \begin{matrix} \text{(B2)} \\ 0 \end{matrix} &
\begin{matrix} \text{(C3)} \\ 1 \end{matrix} & 0 & 1 & 1 & 1\\ \hline
\begin{smallmatrix}i = 3 \end{smallmatrix}& 
\begin{matrix} \text{(B3)} \\ 0 \end{matrix} &  \begin{matrix} \text{(A4)} \\ 1 \end{matrix} & 
\begin{matrix} \text{(C5)} \\ 0 \end{matrix} &  \begin{matrix} \text{(B3)} \\ 0 \end{matrix} &
\begin{matrix} \text{(C4)} \\ 0 \end{matrix} & 0 & 0 & 1 & 0\\ \hline
\begin{smallmatrix}i \ge 4 \end{smallmatrix}& 
\begin{matrix} \text{(B3)} \\ 0 \end{matrix} &  \begin{matrix} \text{(A4)} \\ 1 \end{matrix} & 
\begin{matrix} \text{(C5)} \\ 0 \end{matrix} &  \begin{matrix} \text{(B3)} \\ 0 \end{matrix} &
\begin{matrix} \text{(C5)} \\ 0 \end{matrix} & 0 & 0 & 1 & 0\\ \hline
\end{array}
\] 

(6) Pentagon relation for $\{i, j, m, l, k\}$: 
\[\begin{array}{|l|ccccc|cccc|}
\hline
& \begin{matrix} x_1 \\ (ijml) \end{matrix} & 
\begin{matrix} x_2 \\ (ijlk) \end{matrix} & 
\begin{matrix} x_3 \\ (jklm) \end{matrix} & 
\begin{matrix} x_4 \\ (ijmk) \end{matrix} & 
\begin{matrix} x_5 \\ (iklm) \end{matrix} & x_1 \circ x_2 & x_1 \circ x_3 & x_2 \circ x_3 & x_5 \circ x_4\\\hline
\begin{smallmatrix} i=1 \\ j=2 \\ k=3 \end{smallmatrix} & 
\begin{matrix} \text{(G3)} \\ 0 \end{matrix} &  \begin{matrix} \text{(G3)} \\ 0 \end{matrix} & 
\begin{matrix} \text{(A3)} \\ 0 \end{matrix} &  \begin{matrix} \text{(G3)} \\ 0 \end{matrix} &
\begin{matrix} \text{(A2)} \\ 1 \end{matrix} & 1 & 1 & 0 & 1\\ \hline
\begin{smallmatrix} i=1 \\ j \ge 3 \end{smallmatrix} & 
\begin{matrix} \text{(G3)} \\ 0 \end{matrix} &  \begin{matrix} \text{(G3)} \\ 0 \end{matrix} & 
\begin{matrix} \text{(A4)} \\ 1 \end{matrix} &  \begin{matrix} \text{(G3)} \\ 0 \end{matrix} &
\begin{matrix} \text{(A2)} \\ 1 \end{matrix} & 1 & 0 & 1 & 0\\ \hline
\begin{smallmatrix}i=2 \\ j =3 \end{smallmatrix}& 
\begin{matrix} \text{(B1)} \\ 0 \end{matrix} &  \begin{matrix} \text{(B1)} \\ 0 \end{matrix} & 
\begin{matrix} \text{(A4)} \\ 1 \end{matrix} &  \begin{matrix} \text{(B1)} \\ 0 \end{matrix} &
\begin{matrix} \text{(A3)} \\ 0 \end{matrix} & 1 & 1 & 1 & 0\\ \hline
\begin{smallmatrix}i=2 \\ j \ge 4 \end{smallmatrix}& 
\begin{matrix} \text{(B2)} \\ 0 \end{matrix} &  \begin{matrix} \text{(B2)} \\ 0 \end{matrix} & 
\begin{matrix} \text{(A4)} \\ 1 \end{matrix} &  \begin{matrix} \text{(B2)} \\ 0 \end{matrix} &
\begin{matrix} \text{(A3)} \\ 0 \end{matrix} & 1 & 1 & 1 & 0\\ \hline
\begin{smallmatrix}i \ge  3 \end{smallmatrix}& 
\begin{matrix} \text{(B3)} \\ 0 \end{matrix} &  \begin{matrix} \text{(B3)} \\ 0 \end{matrix} & 
\begin{matrix} \text{(A4)} \\ 1 \end{matrix} &  \begin{matrix} \text{(B3)} \\ 0 \end{matrix} &
\begin{matrix} \text{(A4)} \\ 1 \end{matrix} & 1 & 0 & 1 & 0\\ \hline
\end{array}
\] 

(7) Pentagon relation for $\{i, k, j, l, m\}$: 
\[\begin{array}{|l|ccccc|cccc|}
\hline
& \begin{matrix} x_1 \\ (ikjl) \end{matrix} & 
\begin{matrix} x_2 \\ (iklm) \end{matrix} & 
\begin{matrix} x_3 \\ (jkml) \end{matrix} & 
\begin{matrix} x_4 \\ (ikjm) \end{matrix} & 
\begin{matrix} x_5 \\ (ijlm) \end{matrix} & x_1 \circ x_2 & x_1 \circ x_3 & x_2 \circ x_3 & x_5 \circ x_4\\\hline
\begin{smallmatrix} i=1 \\ j=2 \\ k=3 \end{smallmatrix} & 
\begin{matrix} \text{(G2)} \\ 0 \end{matrix} &  \begin{matrix} \text{(A2)} \\ 1 \end{matrix} & 
\begin{matrix} \text{(B1)} \\ 0 \end{matrix} &  \begin{matrix} \text{(G2)} \\ 0 \end{matrix} &
\begin{matrix} \text{(A1)} \\ 1 \end{matrix} & 0 & 0 & 0 & 0\\ \hline
\begin{smallmatrix} i=1 \\ j=2 \\ k\ge 4 \end{smallmatrix} & 
\begin{matrix} \text{(G2)} \\ 0 \end{matrix} &  \begin{matrix} \text{(A2)} \\ 1 \end{matrix} & 
\begin{matrix} \text{(B2)} \\ 0 \end{matrix} &  \begin{matrix} \text{(G2)} \\ 0 \end{matrix} &
\begin{matrix} \text{(A1)} \\ 1 \end{matrix} & 0 & 0 & 0 & 0\\ \hline
\begin{smallmatrix} i=1 \\ j = 3 \end{smallmatrix} & 
\begin{matrix} \text{(C1)} \\ 0 \end{matrix} &  \begin{matrix} \text{(A2)} \\ 1 \end{matrix} & 
\begin{matrix} \text{(B3)} \\ 0 \end{matrix} &  \begin{matrix} \text{(C1)} \\ 0 \end{matrix} &
\begin{matrix} \text{(A2)} \\ 1 \end{matrix} & 1 & 0 & 1 & 0\\ \hline
\begin{smallmatrix}i=1 \\ j \ge 4 \end{smallmatrix}& 
\begin{matrix} \text{(C2)} \\ 1 \end{matrix} &  \begin{matrix} \text{(A2)} \\ 1 \end{matrix} & 
\begin{matrix} \text{(B3)} \\ 0 \end{matrix} &  \begin{matrix} \text{(C2)} \\ 1 \end{matrix} &
\begin{matrix} \text{(A2)} \\ 1 \end{matrix} & 1 & 0 & 1 & 0\\ \hline
\begin{smallmatrix}i=2 \end{smallmatrix}& 
\begin{matrix} \text{(C3)} \\ 1 \end{matrix} &  \begin{matrix} \text{(A3)} \\ 0 \end{matrix} & 
\begin{matrix} \text{(B3)} \\ 0 \end{matrix} &  \begin{matrix} \text{(C3)} \\ 1 \end{matrix} &
\begin{matrix} \text{(A3)} \\ 0 \end{matrix} & 0 & 1 & 0 & 1\\ \hline
\begin{smallmatrix}i = 3 \end{smallmatrix}& 
\begin{matrix} \text{(C4)} \\ 0 \end{matrix} &  \begin{matrix} \text{(A4)} \\ 1 \end{matrix} & 
\begin{matrix} \text{(B3)} \\ 0 \end{matrix} &  \begin{matrix} \text{(C4)} \\ 0 \end{matrix} &
\begin{matrix} \text{(A4)} \\ 1 \end{matrix} & 0 & 1 & 1 & 0\\ \hline
\begin{smallmatrix}i \ge 4 \end{smallmatrix}& 
\begin{matrix} \text{(C5)} \\ 0 \end{matrix} &  \begin{matrix} \text{(A4)} \\ 1 \end{matrix} & 
\begin{matrix} \text{(B3)} \\ 0 \end{matrix} &  \begin{matrix} \text{(C5)} \\ 0 \end{matrix} &
\begin{matrix} \text{(A4)} \\ 1 \end{matrix} & 0 & 1 & 1 & 0\\ \hline
\end{array}
\] 

(8) Pentagon relation for $\{i, k, j, m, l\}$: 
\[\begin{array}{|l|ccccc|cccc|}
\hline
& \begin{matrix} x_1 \\ (ikjm) \end{matrix} & 
\begin{matrix} x_2 \\ (ikml) \end{matrix} & 
\begin{matrix} x_3 \\ (jklm) \end{matrix} & 
\begin{matrix} x_4 \\ (ikjl) \end{matrix} & 
\begin{matrix} x_5 \\ (ijml) \end{matrix} & x_1 \circ x_2 & x_1 \circ x_3 & x_2 \circ x_3 & x_5 \circ x_4\\\hline
\begin{smallmatrix} i=1 \\ j=2 \end{smallmatrix} & 
\begin{matrix} \text{(G2)} \\ 0 \end{matrix} &  \begin{matrix} \text{(G3)} \\ 0 \end{matrix} & 
\begin{matrix} \text{(A3)} \\ 0 \end{matrix} &  \begin{matrix} \text{(G2)} \\ 0 \end{matrix} &
\begin{matrix} \text{(G3)} \\ 0 \end{matrix} & 0 & 0 & 1 & 1\\ \hline
\begin{smallmatrix} i=1 \\ j=3 \end{smallmatrix} & 
\begin{matrix} \text{(C1)} \\ 0 \end{matrix} &  \begin{matrix} \text{(G3)} \\ 0 \end{matrix} & 
\begin{matrix} \text{(A4)} \\ 1 \end{matrix} &  \begin{matrix} \text{(C1)} \\ 0 \end{matrix} &
\begin{matrix} \text{(G3)} \\ 0 \end{matrix} & 0 & 1 & 1 & 1\\ \hline
\begin{smallmatrix} i=1 \\ j \ge 4 \end{smallmatrix} & 
\begin{matrix} \text{(C2)} \\ 1 \end{matrix} &  \begin{matrix} \text{(G3)} \\ 0 \end{matrix} & 
\begin{matrix} \text{(A4)} \\ 1 \end{matrix} &  \begin{matrix} \text{(C2)} \\ 1 \end{matrix} &
\begin{matrix} \text{(G3)} \\ 0 \end{matrix} & 0 & 1 & 1 & 1\\ \hline
\begin{smallmatrix}i=2 \\ j =3 \end{smallmatrix}& 
\begin{matrix} \text{(C3)} \\ 1 \end{matrix} &  \begin{matrix} \text{(B2)} \\ 0 \end{matrix} & 
\begin{matrix} \text{(A4)} \\ 1 \end{matrix} &  \begin{matrix} \text{(C3)} \\ 1 \end{matrix} &
\begin{matrix} \text{(B1)} \\ 0 \end{matrix} & 1 & 1 & 0 & 1\\ \hline
\begin{smallmatrix}i=2 \\ j \ge 4\end{smallmatrix}& 
\begin{matrix} \text{(C3)} \\ 1 \end{matrix} &  \begin{matrix} \text{(B2)} \\ 0 \end{matrix} & 
\begin{matrix} \text{(A4)} \\ 1 \end{matrix} &  \begin{matrix} \text{(C3)} \\ 1 \end{matrix} &
\begin{matrix} \text{(B2)} \\ 0 \end{matrix} & 1 & 1 & 0 & 1\\ \hline
\begin{smallmatrix}i = 3 \end{smallmatrix}& 
\begin{matrix} \text{(C4)} \\ 0 \end{matrix} &  \begin{matrix} \text{(B3)} \\ 0 \end{matrix} & 
\begin{matrix} \text{(A4)} \\ 1 \end{matrix} &  \begin{matrix} \text{(C4)} \\ 0 \end{matrix} &
\begin{matrix} \text{(B3)} \\ 0 \end{matrix} & 0 & 0 & 1 & 0\\ \hline
\begin{smallmatrix}i \ge 4 \end{smallmatrix}& 
\begin{matrix} \text{(C5)} \\ 0 \end{matrix} &  \begin{matrix} \text{(B3)} \\ 0 \end{matrix} & 
\begin{matrix} \text{(A4)} \\ 1 \end{matrix} &  \begin{matrix} \text{(C5)} \\ 0 \end{matrix} &
\begin{matrix} \text{(B3)} \\ 0 \end{matrix} & 0 & 0 & 1 & 0\\ \hline
\end{array}
\] 

(9) Pentagon relation for $\{i, k, l, j, m\}$: 
\[\begin{array}{|l|ccccc|cccc|}
\hline
& \begin{matrix} x_1 \\ (ijlk) \end{matrix} & 
\begin{matrix} x_2 \\ (ikjm) \end{matrix} & 
\begin{matrix} x_3 \\ (jlkm) \end{matrix} & 
\begin{matrix} x_4 \\ (iklm) \end{matrix} & 
\begin{matrix} x_5 \\ (iljm) \end{matrix} & x_1 \circ x_2 & x_1 \circ x_3 & x_2 \circ x_3 & x_5 \circ x_4\\\hline
\begin{smallmatrix} i=1 \\ j=2 \end{smallmatrix} & 
\begin{matrix} \text{(G3)} \\ 0 \end{matrix} &  \begin{matrix} \text{(G2)} \\ 0 \end{matrix} & 
\begin{matrix} \text{(C3)} \\ 1 \end{matrix} &  \begin{matrix} \text{(A2)} \\ 1 \end{matrix} &
\begin{matrix} \text{(G2)} \\ 0 \end{matrix} & 0 & 0 & 1 & 1\\ \hline
\begin{smallmatrix} i=1 \\ j=3 \end{smallmatrix} & 
\begin{matrix} \text{(G3)} \\ 0 \end{matrix} &  \begin{matrix} \text{(C1)} \\ 0 \end{matrix} & 
\begin{matrix} \text{(C4)} \\ 0 \end{matrix} &  \begin{matrix} \text{(A2)} \\ 1 \end{matrix} &
\begin{matrix} \text{(C1)} \\ 0 \end{matrix} & 0 & 1 & 1 & 1\\ \hline
\begin{smallmatrix} i=1 \\ j \ge 4 \end{smallmatrix} & 
\begin{matrix} \text{(G3)} \\ 0 \end{matrix} &  \begin{matrix} \text{(C2)} \\ 1 \end{matrix} & 
\begin{matrix} \text{(C5)} \\ 0 \end{matrix} &  \begin{matrix} \text{(A2)} \\ 1 \end{matrix} &
\begin{matrix} \text{(C2)} \\ 1 \end{matrix} & 0 & 1 & 1 & 1\\ \hline
\begin{smallmatrix}i=2 \\ j =3 \end{smallmatrix}& 
\begin{matrix} \text{(B1)} \\ 0 \end{matrix} &  \begin{matrix} \text{(C3)} \\ 1 \end{matrix} & 
\begin{matrix} \text{(C4)} \\ 0 \end{matrix} &  \begin{matrix} \text{(A3)} \\ 0 \end{matrix} &
\begin{matrix} \text{(C3)} \\ 1 \end{matrix} & 1 & 0 & 1 & 0\\ \hline
\begin{smallmatrix}i=2 \\ j \ge 4\end{smallmatrix}& 
\begin{matrix} \text{(B2)} \\ 0 \end{matrix} &  \begin{matrix} \text{(C3)} \\ 1 \end{matrix} & 
\begin{matrix} \text{(C5)} \\ 0 \end{matrix} &  \begin{matrix} \text{(A3)} \\ 0 \end{matrix} &
\begin{matrix} \text{(C3)} \\ 1 \end{matrix} & 1 & 0 & 1 & 0\\ \hline
\begin{smallmatrix}i = 3 \end{smallmatrix}& 
\begin{matrix} \text{(B3)} \\ 0 \end{matrix} &  \begin{matrix} \text{(C4)} \\ 0 \end{matrix} & 
\begin{matrix} \text{(C5)} \\ 0 \end{matrix} &  \begin{matrix} \text{(A4)} \\ 1 \end{matrix} &
\begin{matrix} \text{(C4)} \\ 0 \end{matrix} & 1 & 1 & 1 & 0\\ \hline
\begin{smallmatrix}i \ge 4 \end{smallmatrix}& 
\begin{matrix} \text{(B3)} \\ 0 \end{matrix} &  \begin{matrix} \text{(C5)} \\ 0 \end{matrix} & 
\begin{matrix} \text{(C5)} \\ 0 \end{matrix} &  \begin{matrix} \text{(A4)} \\ 1 \end{matrix} &
\begin{matrix} \text{(C5)} \\ 0 \end{matrix} & 1 & 1 & 1 & 0\\ \hline
\end{array}
\] 

(10) Pentagon relation for $\{i, k, m, j, l\}$: 
\[\begin{array}{|l|ccccc|cccc|}
\hline
& \begin{matrix} x_1 \\ (ijlk) \end{matrix} & 
\begin{matrix} x_2 \\ (ikjm) \end{matrix} & 
\begin{matrix} x_3 \\ (jlkm) \end{matrix} & 
\begin{matrix} x_4 \\ (iklm) \end{matrix} & 
\begin{matrix} x_5 \\ (iljm) \end{matrix} & x_1 \circ x_2 & x_1 \circ x_3 & x_2 \circ x_3 & x_5 \circ x_4\\\hline
\begin{smallmatrix} i=1 \\ j=2 \end{smallmatrix} & 
\begin{matrix} \text{(G3)} \\ 0 \end{matrix} &  \begin{matrix} \text{(G2)} \\ 0 \end{matrix} & 
\begin{matrix} \text{(C3)} \\ 1 \end{matrix} &  \begin{matrix} \text{(G3)} \\ 0 \end{matrix} &
\begin{matrix} \text{(G2)} \\ 0 \end{matrix} & 1 & 1 & 0 & 1\\ \hline
\begin{smallmatrix} i=1 \\ j=3 \end{smallmatrix} & 
\begin{matrix} \text{(G3)} \\ 0 \end{matrix} &  \begin{matrix} \text{(C1)} \\ 0 \end{matrix} & 
\begin{matrix} \text{(C4)} \\ 0 \end{matrix} &  \begin{matrix} \text{(G3)} \\ 0 \end{matrix} &
\begin{matrix} \text{(C1)} \\ 0 \end{matrix} & 1 & 1 & 1 & 1\\ \hline
\begin{smallmatrix} i=1 \\ j \ge 4 \end{smallmatrix} & 
\begin{matrix} \text{(G3)} \\ 0 \end{matrix} &  \begin{matrix} \text{(C2)} \\ 1 \end{matrix} & 
\begin{matrix} \text{(C5)} \\ 0 \end{matrix} &  \begin{matrix} \text{(G3)} \\ 0 \end{matrix} &
\begin{matrix} \text{(C2)} \\ 1 \end{matrix} & 1 & 1 & 1 & 1\\ \hline
\begin{smallmatrix}i=2 \\ j =3 \end{smallmatrix}& 
\begin{matrix} \text{(B1)} \\ 0 \end{matrix} &  \begin{matrix} \text{(C3)} \\ 1 \end{matrix} & 
\begin{matrix} \text{(C4)} \\ 0 \end{matrix} &  \begin{matrix} \text{(B2)} \\ 0 \end{matrix} &
\begin{matrix} \text{(C3)} \\ 1 \end{matrix} & 0 & 1 & 0 & 1\\ \hline
\begin{smallmatrix}i=2 \\ j \ge 4\end{smallmatrix}& 
\begin{matrix} \text{(B2)} \\ 0 \end{matrix} &  \begin{matrix} \text{(C3)} \\ 1 \end{matrix} & 
\begin{matrix} \text{(C5)} \\ 0 \end{matrix} &  \begin{matrix} \text{(B2)} \\ 0 \end{matrix} &
\begin{matrix} \text{(C3)} \\ 1 \end{matrix} & 0 & 1 & 0 & 1\\ \hline
\begin{smallmatrix}i = 3 \end{smallmatrix}& 
\begin{matrix} \text{(B3)} \\ 0 \end{matrix} &  \begin{matrix} \text{(C4)} \\ 0 \end{matrix} & 
\begin{matrix} \text{(C5)} \\ 0 \end{matrix} &  \begin{matrix} \text{(B3)} \\ 0 \end{matrix} &
\begin{matrix} \text{(C4)} \\ 0 \end{matrix} & 0 & 0 & 1 & 1\\ \hline
\begin{smallmatrix}i \ge 4 \end{smallmatrix}& 
\begin{matrix} \text{(B3)} \\ 0 \end{matrix} &  \begin{matrix} \text{(C5)} \\ 0 \end{matrix} & 
\begin{matrix} \text{(C5)} \\ 0 \end{matrix} &  \begin{matrix} \text{(B3)} \\ 0 \end{matrix} &
\begin{matrix} \text{(C5)} \\ 0 \end{matrix} & 0 & 0 & 1 & 1\\ \hline
\end{array}
\] 

(11) Pentagon relation for $\{i, l, j, k, m\}$: 
\[\begin{array}{|l|ccccc|cccc|}
\hline
& \begin{matrix} x_1 \\ (ikjl) \end{matrix} & 
\begin{matrix} x_2 \\ (ilkm) \end{matrix} & 
\begin{matrix} x_3 \\ (jkml) \end{matrix} & 
\begin{matrix} x_4 \\ (iljm) \end{matrix} & 
\begin{matrix} x_5 \\ (ijkm) \end{matrix} & x_1 \circ x_2 & x_1 \circ x_3 & x_2 \circ x_3 & x_5 \circ x_4\\\hline
\begin{smallmatrix} i=1 \\ j=2 \\ k=3\end{smallmatrix} & 
\begin{matrix} \text{(G2)} \\ 0 \end{matrix} &  \begin{matrix} \text{(C1)} \\ 0 \end{matrix} & 
\begin{matrix} \text{(B1)} \\ 0 \end{matrix} &  \begin{matrix} \text{(G2)} \\ 0 \end{matrix} &
\begin{matrix} \text{(G1)} \\ 0 \end{matrix} & 0 & 0 & 0 & 0\\ \hline
\begin{smallmatrix} i=1 \\ j=2 \\ k \ge 4 \end{smallmatrix} & 
\begin{matrix} \text{(G2)} \\ 0 \end{matrix} &  \begin{matrix} \text{(C2)} \\ 1 \end{matrix} & 
\begin{matrix} \text{(B2)} \\ 0 \end{matrix} &  \begin{matrix} \text{(G2)} \\ 0 \end{matrix} &
\begin{matrix} \text{(A1)} \\ 1 \end{matrix} & 0 & 0 & 0 & 0\\ \hline
\begin{smallmatrix} i=1 \\ j=3 \end{smallmatrix} & 
\begin{matrix} \text{(C1)} \\ 0 \end{matrix} &  \begin{matrix} \text{(C2)} \\ 1 \end{matrix} & 
\begin{matrix} \text{(B3)} \\ 0 \end{matrix} &  \begin{matrix} \text{(C1)} \\ 0 \end{matrix} &
\begin{matrix} \text{(A2)} \\ 1 \end{matrix} & 1 & 0 & 0 & 1\\ \hline
\begin{smallmatrix}i=1 \\ j \ge 4 \end{smallmatrix}& 
\begin{matrix} \text{(C2)} \\ 1 \end{matrix} &  \begin{matrix} \text{(C2)} \\ 1 \end{matrix} & 
\begin{matrix} \text{(B3)} \\ 0 \end{matrix} &  \begin{matrix} \text{(C2)} \\ 1 \end{matrix} &
\begin{matrix} \text{(A2)} \\ 1 \end{matrix} & 1 & 0 & 0 & 1\\ \hline
\begin{smallmatrix}i=2 \end{smallmatrix}& 
\begin{matrix} \text{(C3)} \\ 1 \end{matrix} &  \begin{matrix} \text{(C3)} \\ 1 \end{matrix} & 
\begin{matrix} \text{(B3)} \\ 0 \end{matrix} &  \begin{matrix} \text{(C3)} \\ 1 \end{matrix} &
\begin{matrix} \text{(A3)} \\ 0 \end{matrix} & 0 & 1 & 1 & 1\\ \hline
\begin{smallmatrix}i = 3 \end{smallmatrix}& 
\begin{matrix} \text{(C4)} \\ 0 \end{matrix} &  \begin{matrix} \text{(C4)} \\ 0 \end{matrix} & 
\begin{matrix} \text{(B3)} \\ 0 \end{matrix} &  \begin{matrix} \text{(C4)} \\ 0 \end{matrix} &
\begin{matrix} \text{(A4)} \\ 1 \end{matrix} & 0 & 1 & 1 & 1\\ \hline
\begin{smallmatrix}i \ge 4 \end{smallmatrix}& 
\begin{matrix} \text{(C5)} \\ 0 \end{matrix} &  \begin{matrix} \text{(C5)} \\ 0 \end{matrix} & 
\begin{matrix} \text{(B3)} \\ 0 \end{matrix} &  \begin{matrix} \text{(C5)} \\ 0 \end{matrix} &
\begin{matrix} \text{(A4)} \\ 1 \end{matrix} & 0 & 1 & 1 & 1\\ \hline
\end{array}
\] 

(12) Pentagon relation for $\{i, l, k, j, m\}$: 
\[\begin{array}{|l|ccccc|cccc|}
\hline
& \begin{matrix} x_1 \\ (ijkl) \end{matrix} & 
\begin{matrix} x_2 \\ (iljm) \end{matrix} & 
\begin{matrix} x_3 \\ (jklm) \end{matrix} & 
\begin{matrix} x_4 \\ (ilkm) \end{matrix} & 
\begin{matrix} x_5 \\ (ikjm) \end{matrix} & x_1 \circ x_2 & x_1 \circ x_3 & x_2 \circ x_3 & x_5 \circ x_4\\\hline
\begin{smallmatrix} i=1 \\ j=2 \\ k=3\end{smallmatrix} & 
\begin{matrix} \text{(G1)} \\ 0 \end{matrix} &  \begin{matrix} \text{(G2)} \\ 0 \end{matrix} & 
\begin{matrix} \text{(A3)} \\ 0 \end{matrix} &  \begin{matrix} \text{(C1)} \\ 0 \end{matrix} &
\begin{matrix} \text{(G2)} \\ 0 \end{matrix} & 0 & 0 & 1 & 1\\ \hline
\begin{smallmatrix} i=1 \\ j=2 \\ k \ge 4 \end{smallmatrix} & 
\begin{matrix} \text{(A1)} \\ 1 \end{matrix} &  \begin{matrix} \text{(G2)} \\ 0 \end{matrix} & 
\begin{matrix} \text{(A3)} \\ 0 \end{matrix} &  \begin{matrix} \text{(C2)} \\ 1 \end{matrix} &
\begin{matrix} \text{(G2)} \\ 0 \end{matrix} & 0 & 0 & 1 & 1\\ \hline
\begin{smallmatrix} i=1 \\ j=3 \end{smallmatrix} & 
\begin{matrix} \text{(A2)} \\ 1 \end{matrix} &  \begin{matrix} \text{(C1)} \\ 0 \end{matrix} & 
\begin{matrix} \text{(A4)} \\ 1 \end{matrix} &  \begin{matrix} \text{(C2)} \\ 1 \end{matrix} &
\begin{matrix} \text{(C1)} \\ 0 \end{matrix} & 0 & 0 & 0 & 1\\ \hline
\begin{smallmatrix}i=1 \\ j \ge 4 \end{smallmatrix}& 
\begin{matrix} \text{(A2)} \\ 1 \end{matrix} &  \begin{matrix} \text{(C2)} \\ 1 \end{matrix} & 
\begin{matrix} \text{(A4)} \\ 1 \end{matrix} &  \begin{matrix} \text{(C2)} \\ 1 \end{matrix} &
\begin{matrix} \text{(C2)} \\ 1 \end{matrix} & 0 & 0 & 0 & 1\\ \hline
\begin{smallmatrix}i=2 \end{smallmatrix}& 
\begin{matrix} \text{(A3)} \\ 0 \end{matrix} &  \begin{matrix} \text{(C3)} \\ 1 \end{matrix} & 
\begin{matrix} \text{(A4)} \\ 1 \end{matrix} &  \begin{matrix} \text{(C3)} \\ 1 \end{matrix} &
\begin{matrix} \text{(C3)} \\ 1 \end{matrix} & 1 & 0 & 1 & 0 \\ \hline
\begin{smallmatrix}i = 3 \end{smallmatrix}& 
\begin{matrix} \text{(A4)} \\ 1 \end{matrix} &  \begin{matrix} \text{(C4)} \\ 0 \end{matrix} & 
\begin{matrix} \text{(A4)} \\ 1 \end{matrix} &  \begin{matrix} \text{(C4)} \\ 0 \end{matrix} &
\begin{matrix} \text{(C4)} \\ 0 \end{matrix} & 1 & 1 & 1 & 1\\ \hline
\begin{smallmatrix}i \ge 4 \end{smallmatrix}& 
\begin{matrix} \text{(A4)} \\ 1 \end{matrix} &  \begin{matrix} \text{(C5)} \\ 0 \end{matrix} & 
\begin{matrix} \text{(A4)} \\ 1 \end{matrix} &  \begin{matrix} \text{(C5)} \\ 0 \end{matrix} &
\begin{matrix} \text{(C5)} \\ 0 \end{matrix} & 1 & 1 & 1 & 1\\ \hline
\end{array}
\] 
This completes the proof of Lemma \ref{lem:relpentagon}. 
\end{proof}

\begin{proof}[Proof of Proposition $\ref{prop:representation}$] 
Lemmas \ref{lem:relinvol}, \ref{lem:relcomm}, \ref{lem:relpentagon} show that our assignment $\mathfrak{e} (X)$ 
to $X \in \mathcal{G}_n^d$ gives a well-defined homomorphism $r \colon \Gamma_n^4 \to 
(\mathbb{Z}/2\mathbb{Z}) \widetilde{\times} H_1 (\Gamma_n^4) \subset \operatorname{GL}_{N_n+2}(\mathbb{Z}/2\mathbb{Z})$. 
Also we have 
\[r(c)=[r((1234)), r((1235))]=((1234) \circ (1235)+ (1235) \circ (1234), 0) = (0+1,0)=(1,0).\]
This completes the proof. 
\end{proof}

\section{A computer-aided proof for Theorem \ref{thm:main}}\label{sec:computer}

This section gives a computer-aided proof for Theorem \ref{thm:main}. Indeed, 
we may partially use computer calculations to prove both Propositions \ref{prop:finitequotient} and \ref{prop:representation}. 
This proof is the one the authors first obtained. 
Here we write it as another proof, which might be useful from a computational group-theoretical point of view. 

We first prove Proposition \ref{prop:representation}. It is a standard task to check 
that the assignment $\mathfrak{e} (X)$ to 
$X \in \mathcal{G}_n^d$ defined in Section \ref{sec:nonabelian} gives a well-defined homomorphism    
$r \colon \Gamma_n^4 \longrightarrow (\mathbb{Z}/2\mathbb{Z}) \widetilde{\times} H_1 (\Gamma_n^4)$ 
for $n=11$. It does not take much time by a (usual and personal) computer. 
This checking process includes the cases where $n=6,7,8,9,10$. 

Note that our choice of the minimal generating set $\Lambda_n$ as well as the total order $\prec$ and 
the classification of the elements in $\mathcal{G}_n^d$ given in Lemma \ref{lem:bymingen} 
are compatible with respect to the natural homomorphism 
\[\Gamma_{11}^4 \longrightarrow \Gamma_n^4, \qquad (stuv) \longmapsto (f(s) f(t) f(u) f(v))\]
which is induced from a monotonically increasing map $f \colon [11] \to [n]$ fixing $\{1, 2, 3\}$. 
Indeed, checking $r([(4567), (89\,10\,11)])=(0,0)$ costs the largest $n=11$. 
Therefore checking only for the case where $n=11$ suffices to check for all the cases where $n \ge 6$, which 
proves Proposition \ref{prop:representation}. It is easily checked that 
\[r([(1234), (1235)])=r([(1234), (1253)])=r([(1234), (1325)])=(1,0),\]
which implies that $[(1234), (1235)]$, $[(1234), (1253)]$ and $[(1234), (1325)]$ are all non-trivial in $\Gamma_n^4$ 
for all $n \ge 6$. 

Next, we prove Proposition \ref{prop:finitequotient}. 
We may use a computer to show that the group $\Gamma_6^4$ is a finite group of order $2^{20}=1,048,576$. 
The authors first did this computation by using GAP, which took about 30 minutes, and then 
implemented the coset enumeration algorithm, following the one displayed in the Handbook \cite{Handbook}, 
to get the coset table for the pair $(\Gamma_6^4, \{1\})$ by Mathematica, 
which took some days. In our previous paper \cite[Theorem 3.2]{DTS}, we showed that 
the order of the abelianization $H_1 (\Gamma_6^4)$ is $2^{19}=524,288$. Hence 
we have $[\Gamma_6^4, \Gamma_6^4] \cong \mathbb{Z}/2\mathbb{Z}$. It immediately follows 
that $c:=[(1234), (1235)]=[(1234), (1253)]=[(1234), (1325)] \in \Gamma_6^4$ gives 
the generator of $[\Gamma_6^4, \Gamma_6^4]$. 
Since the subgroup $[\Gamma_6^4, \Gamma_6^4]$ is characteristic, it is fixed by the action of $\mathfrak{S}_6$. 
In particular, $c$ is an $\mathfrak{S}_6$-invariant element. Then 
the equality $[c, (3456)]=[[(1234), (1235)], (3456)]=1$ together with the $\mathfrak{S}_6$-action shows that 
$c$ is in the center of $\Gamma_6^4$. Proposition \ref{prop:finitequotient} for $n = 6$ holds from this. 
The proof of Lemma \ref{lem:comm4} for $n =6$ becomes easier in this line of arguments. 
The pentagon relation gives
\begin{align*}
(1234)(1243)&=(1234) \cdot (1435)(1245)(2435)(1235)\\
&=c^4(1435)(1245)(2435)(1235) \cdot (1234)=(1243)(1234).
\end{align*}
\noindent
The other cases are similarly checked. 

For general $n \ge 6$, the equalities 
\[[(1234), (1235)]=[(1234), (1236)]=[(1235), (1236)]\]
hold in $\Gamma_n^4$ since they do in $\Gamma_6^4$. 
We set $c=[(1234), (1235)] \in \Gamma_n^4$ again. 
Applying $(6k) \in \mathfrak{S}_n$ for a fixed $k \ge 7$, we have 
\[c=[(1234), (1235)]=[(1234), (123k)]=[(1235), (123k)].\] 
Then applying $(5l) \in \mathfrak{S}_n$ for a fixed $l \ge 7$ with $l \neq k$, we have 
$c=[(1234), (123k)]=[(123l), (123k)]$. Finally, applying $(6l) \in \mathfrak{S}_n$ gives $c=[(1234), (123k)]=[(1236), (123k)]$, 
which shows that $[(123l), (123k)]$ does not depend on the choices of $k, l \ge 4$ with $k \neq l$. 
Since $c=[(1234), (1235)]=[(1234), (1253)]=[(1234), (1263)]$ and 
$c=[(1234), (1235)]=[(1234), (1325)]=[(1234), (1326)]$, we can similarly derive that 
\[c=[(123k), (12l3)]=[(123k), (132l)]\]
for any $k, l \ge 4$ with $k \neq l$. 

Applying the natural homomorphism 
\[\Gamma_6^4 \longrightarrow \Gamma_n^4, \qquad (stuv) \longmapsto (g(s) g(t) g(u) g(v))\]
induced from the map $g(m)=m$ for $m=1,2,3,4,5$ and $g(6)=k \ge 7$ to 
\[c=[(1234), (1235)]=[(1264), (1265)]=[(1264), (1256)]=[(1264), (1625)] \in \Gamma_6^4,\] 
we have 
\[c=[(1234), (1235)]=[(12k4), (12k5)]=[(12k4), (125k)]=[(12k4), (1k25)] \in \Gamma_n^4.\]  
Continuing in this way, we can show Proposition \ref{prop:finitequotient} 
and Lemma \ref{lem:comm4} for $n \ge 7$.

\section{Remarks}\label{sec:remarks}

Finally, we discuss possible future directions of our study. 

\begin{enumerate}
\item Besides the definition of $\Gamma_n^4$ in \cite{KM}, Kim and Manturov 
constructed a homomorphism 
\[P_n \longrightarrow \Gamma_n^4\]
from the pure braid group $P_n$ of $n$ strings to $\Gamma_n^4$ by 
using their geometric background of $\Gamma_n^4$. 
It would be possible to use our explicit description of $\Gamma_n^4$ 
to determine the image of the homomorphism. 
Details will be discussed elsewhere. 

\item Determine the center of $\Gamma_n^4$. It is strictly bigger than 
$\mathbb{Z}/2\mathbb{Z}$ generated by $c$. 
Indeed, as proved in Lemma \ref{lem:comm4center}, 
the abelian subgroup generated by $(ijkl)(stuv)$ with $|\{i,j,k,l\} \cap \{s,t,u,v\}| =4$ 
is included in the center. 

\item Our description of $\Gamma_n^4$ as a central extension of $H_1 (\Gamma_n^4)$ by $\mathbb{Z}/2\mathbb{Z}$ 
is stable with respect to increasing $n$. This implies that the natural homomorphism  
$\Gamma_n^4 \to \Gamma_{n+1}^4$ is injective for $n \ge 6$. 
Also, the description makes sense even when we take the direct limit $\displaystyle\lim_{n \to \infty} \Gamma_n^4$. 
The geometric meaning of this central extension might be interesting. 

\item When we consider the relationship between braid groups and symmetric groups, or more generally Artin groups and Coxeter groups, it would be natural to 
introduce the following groups, which had been already discussed in \cite[Sections 2 and 4]{DTS}: 

\begin{definition}\label{def:gammahat}
For $n \ge 4$, the group $\widehat{\Gamma_n^4}$ is defined by the following presentation: 

(Generators) $\{(ijkl) \mid \{i,j,k,l\}\subset [n],\ (i,j,k,l\text{: distinct})\}$

(Relations) There are three types of relations: 
\[\begin{array}{ll}
 (2)&  (ijkl)(stuv)=(stuv)(ijkl),\ (|\{i,j,k,l\} \cap \{s,t,u,v\}| \le 2);\\
 (3)' &  (ijkl)(ijlm)(jklm)(ijkm)^{-1}(iklm)^{-1}=1, \ (i,j,k,l,m \text{ distinct});\\
 (4)' &  (ijkl)=(jkli)^{-1}=(lkji)^{-1}.
\end{array}\]
\end{definition}
\noindent
We call the relations $(3)'$ the {\it signed pentagon relations} and 
call the relations $(4)'$ the {\it signed dihedral relations} (see Figure \ref{fig:geom2}). 
We have a natural projection $\widehat{\Gamma_n^4} \twoheadrightarrow \Gamma_n^4$ 
sending $(ijkl) \in \widehat{\Gamma_n^4}$ to $(ijkl) \in \Gamma_n^4$.  
Study the structure of $\widehat{\Gamma_n^4}$ while comparing it with $\Gamma_n^4$.  

\begin{figure}[htbp]
\begin{center}
\includegraphics[width=0.75\textwidth]{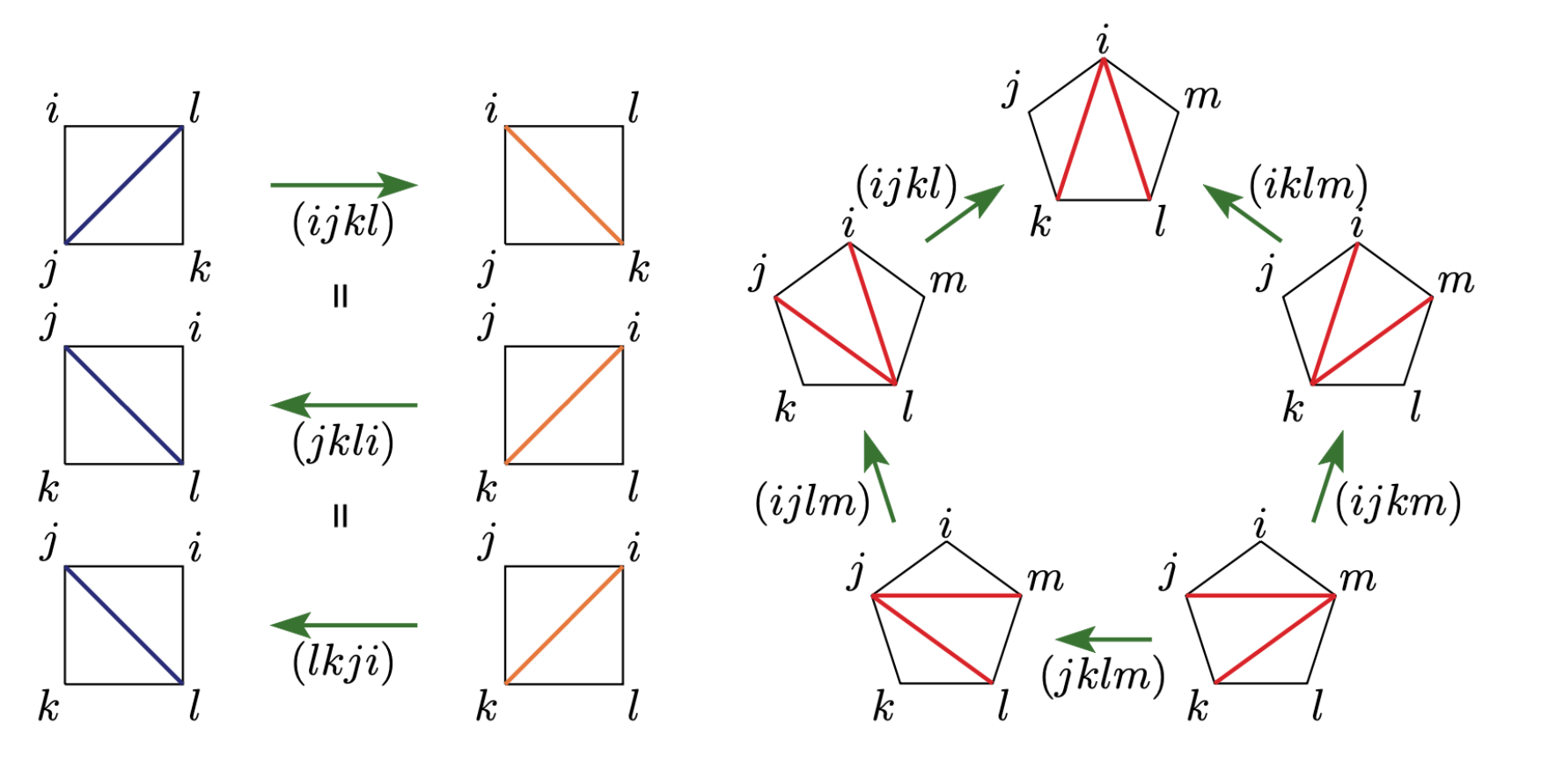}
\caption{Graphical meaning of the relations of $\widehat{\Gamma_n^4}$}
\label{fig:geom2}
\end{center}
\end{figure}

\item In \cite[Section 5]{DTS}, we showed that $\Gamma_5^4$ is a non-abelian infinite group. 
Studying the structure of $\Gamma_5^4$ remains open. 

\item As seen in Theorem \ref{thm:main}, the commutative relations make $\Gamma_n^4$ for $n \ge 6$ finite groups. 
It might be possible to define more interesting groups by dropping some of the commutative relations. 

\item In \cite[Section 6]{DTS}, we also defined the inclusing-order version $\Delta_n^4$ of $\Gamma_n^4$ and 
gave a minimal generating set. 
Our preliminary computation for $\Delta_6^4$ using the coset table for the pair $(\Delta_6^4, \{1\})$ 
says that the group $\Delta_6^4$ 
is a finite group of order $2^{17}=131,072$. Furthermore we observed that $\Delta_6^4$ 
is a $3$-step nilpotent group. Hence, the natural homomorphism $\Delta_6^4 \to \Gamma_6^4$ is not injective. 
Studying the structure of $\Delta_n^4$ might be interesting. 
\end{enumerate}

\bigskip
{\it Acknowledgement} \ 
The authors would like to thank Tomotada Ohtsuki and Seiichi Kamada for their organization of 
the international conference ``The 19th East Asian Conference on Geometric Topology'', when the second, third and fourth 
authors had a chance to meet the first author. 
Also, they would like to thank Vassily Olegovich Manturov for helpful comments which 
included informing them about the paper \cite{Styrt}. 

\bibliographystyle{amsplain}

\end{document}